\pdfoutput=1
\RequirePackage{ifpdf}
\ifpdf
\documentclass[pdftex]{sigma}
\else
\documentclass{sigma}
\fi

\newtheorem{thm}{Theorem}[section]
\newtheorem{cor}[thm]{Corollary}

\newtheorem{conj}[thm]{Conjecture}

\theoremstyle{definition}
\newtheorem{rem}[thm]{Remark}

\usepackage{arydshln}
\newcommand{\C}{{\mathbb C}}
\newcommand\R{{\mathbb R}}
\newcommand\Z{{\mathbb Z}}

\renewcommand\H{{\mathcal H}}

\newcommand\Q{{\mathcal Q}}
\newcommand\cL{{\mathcal L}}

\newcommand\E{{\mathbb E}}
\renewcommand\P{{\mathbb P}}

\def\ed{\stackrel{{\rm d}}{=}}

\begin{document}

\allowdisplaybreaks

\newcommand{\arXivNumber}{2601.06893}

\renewcommand{\PaperNumber}{069}

\FirstPageHeading

\ShortArticleName{Discrete Whittaker Processes and Non-Intersecting Brownian Bridges}

\ArticleName{Absorption Times for Discrete Whittaker Processes\\ and Non-Intersecting Brownian Bridges}

\Author{Neil O'CONNELL}

\AuthorNameForHeading{N.~O'Connell}

\Address{School of Mathematics and Statistics, University College Dublin, Dublin 4, Ireland}
\Email{\mail{neil.oconnell@ucd.ie}}

\ArticleDates{Received February 20, 2026, in final form July 09, 2026; Published online July 20, 2026}

\Abstract{We present evidence for a~conjectural relationship between absorption times for discrete Whittaker processes and
maximal heights of non-intersecting Brownian bridges.}


\Classification{37K10; 60K35; 60B20; 82C23}

\section{Introduction}

It is well known (see, for example,~\cite{bpy} for a~survey)
that twice the square of the maximum of a~reflected Brownian bridge, starting and ending at zero,
has the same law as the random variable%
\begin{equation}\label{S} S=\sum_{n=1}^\infty \frac{e_n}{n^2},\end{equation}
where $e_1, e_2, \dots$ is a~sequence of independent standard exponential random variables, and that twice the square of the
maximum of a~standard Brownian excursion has the same law as $S+S'$, where $S'$ is an independent copy of $S$.
In this paper, we present evidence for a~conjectural generalisation of these identities in law, which relates maximal heights of non-intersecting
reflected Brownian bridges and non-intersecting Brownian excursions to absorption times for the discrete Whittaker processes introduced in~\cite{noc23}.
The present study of the latter is motivated by an attempt to understand the large scale behaviour of these processes,
in particular the question of whether they belong to the KPZ universality class~\cite{c}, which we now conjecture to be the case
based on this apparent connection. There is an extensive literature on maximal heights of non-intersecting Brownian bridges,
their asymptotics, and connections to 2D Yang--Mills theory on the sphere~\cite{f,fms,kat,lie,smcrf}.

Let $\Pi$ be the set of infinite arrays
$(\pi_{ij})_{i,j=1}^\infty$ of nonnegative integers
satisfying
\[ \pi_{ij} \ge \max\{ \pi_{i,j-1},\pi_{i-1,j}\}\]
for all $i,j\ge 1$ with the convention $\pi_{i0}=\pi_{0j}=0$.
We consider a~continuous time Markov chain on $\Pi$ which evolves
as follows: for each pair of indices $(i,j)$, subtract one from $\pi_{ij}$ at rate~${ (\pi_{ij}-\pi_{i,j-1})(\pi_{ij}-\pi_{i-1,j})}$.
A~simulation of this process, started at $\pi_{ij}=70$ for all $i,j\ge 1$ and
run until the first time that $\pi_{1,50}=0$, is shown in its final state in Figure~\ref{figure1}.

\begin{figure}\centering
\includegraphics[scale=0.9]{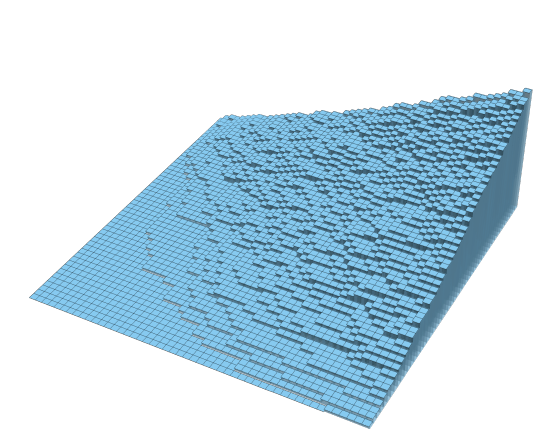}

\caption{Simulation of the Markov chain on $\Pi$, started with $\pi_{ij}=70$ for all $i,j\ge 1$
and run until the first time that $\pi_{1,50}=0$. The height of this surface over the box with
coordinates $(i,j)$ is the value of~$\pi_{ij}$ at this stopping time, shown here for $1\le i,j\le 50$.}\label{figure1}
\end{figure}

We would like to understand the evolution of the zero set $ \{(i,j)\mid \pi_{ij}=0\}$.
This corresponds to the flat region on the left of
Figure~\ref{figure1}.
To this end, let $T_r$ be the first time that $\max_{i=1,\dots,r}\pi_{i,r-i+1}=0$.
Note that the projection of the chain onto the values of $\pi_{ij}$ for~${2\le i+j\le r+1}$
is autonomously Markov and that $T_r$ is the absorption time for this projection.

In~\cite{noc23}, it was shown that the above Markov chain on $\Pi$ is related to the quantum Toda chain
and has a~unique entrance law starting from the array with $\pi_{ij}=+\infty$ for all $i,j\ge 1$.
The latter statement means that there is a~unique family of probability measures $\mu_t$, $t>0$ such that
(i) for all $t,s>0$, if we start the chain with the random initial condition $\mu_t$, then its law at time~$s$ will be $\mu_{t+s}$ and (ii) the corresponding realisation of the chain, which has the law $\mu_t$
at each time $t>0$, will satisfy $\pi_{ij}(t)\to +\infty$ in probability as $t\to 0$ for all $i,j\ge 1$.
In the following, we will assume the Markov chain is started at $+\infty$ and denote its law by~$\P$.

Let us first note that $\P(T_r<\infty)=1$. This can be seen by induction, as follows.
Let $\tau_{ij}$ be the first time that $\pi_{ij}=0$ and note that $T_r=\max_{i=1,\dots,r}\tau_{i,r-i+1}$,
so it suffices to show that~${\P(\tau_{ij}<\infty)=1}$ for all $i,j\ge 1$.
From the definition, the corner value $\pi_{11}$ evolves as a~pure death process
on the nonnegative integers with quadratic rates: it jumps from $n$ to $n-1$ at rate~$n^2$.
The law of $\tau_{11}$ is therefore the same as that of $S$ defined by~\eqref{S}.
Note that, by monotone convergence, $ES=\zeta(2)<\infty$, which implies that $S$ (and hence $\tau_{11}$)
is almost surely finite. Now consider the time $\tau_{ij}$ and note the inequality
\[\tau_{ij}\ge \sigma_{ij}:=\max\{ \tau_{i,j-1},\tau_{i-1,j}\}\]
with the convention $\tau_{i0}=\tau_{0j}=0$.
By the induction hypothesis, $\sigma_{ij}<\infty$ almost surely.
From the time $\sigma_{ij}$ onwards, the value of $\pi_{ij}$ evolves as a~pure death process, jumping from
$n$ to $n-1$ at rate $n^2$; the additional time $\tau_{ij}-\sigma_{ij}$ taken for $\pi_{ij}$ to reach zero is therefore stochastically bounded above by the random variable $S$, and hence $\tau_{ij}<\infty$ almost surely, as required.

The Hamiltonian of the $(r+1)$-particle quantum Toda chain is given by
\[\H^r = -\frac12 \sum_{i=1}^{r+1} \frac{\partial^2}{\partial x_i^2} + \sum_{i=1}^{r}{\rm e}^{x_i-x_{i+1}}.\]
The corresponding eigenvalue equation has series solutions known as
fundamental Whittaker functions~\cite{h,is}.
In particular, if we consider the series
\[
\phi_r(x)=\sum_{n\in\Z_+^r} a_r(n) \prod_{i=1}^r {\rm e}^{n_i(x_i-x_{i+1})}, \]
where $\Z_+$ denotes the nonnegative integers,
then $\H^r\phi_r=0$ provided the coefficients $a_r(n)$ satisfy the recursion
\begin{equation}\label{rec-ar}
\left(\sum_{i=1}^r n_i^2 - \sum_{i=1}^{r-1} n_i n_{i+1}\right) a_r(n) = \sum_{i=1}^r a_r(n-e_i),
\end{equation}
where $e_1,\dots,e_r$ are the standard basis vectors in $\Z^r$ and
with the convention $a_r(n)=0$ for~${n\notin\Z_+^r}$.
If we set $a_r(0)=1$, then this recursion has a~unique solution.
In~\cite{noc23}, it is shown that, for the Markov chain on $\Pi$ started from $+\infty$,
the values of $(\pi_{1,r},\pi_{2,r-1},\dots,\pi_{r,1})$ evolve as a~Markov chain on $\Z_+^r$,
which jumps from $n$ to $n-e_i$ at rate $a_r(n-e_i)/a_r(n)$, for each~${i=1,\dots,r}$.
As $T_r$ is also the absorption time for this projection, we can therefore study its distribution
using the theory of the quantum Toda chain and its eigenfunctions.

The eigenfunctions we need to consider are for a~particular discrete quantisation of the Toda chain
(corresponding to the recursion~\eqref{rec-ar}).
We consider families of left eigenfunctions, defined in terms of Mellin transforms of class one
Whittaker functions, and right eigenfunctions, defined in terms of coefficients
of fundamental Whittaker functions. These are discussed in detail in Section~\ref{ef}.
With these in hand, our approach is to formulate spectral expansions for the transition probabilities
of the chain and hence also the distribution function $\P(T_r\le t)$.
This is technically quite difficult however, as the series
in question are quite singular and difficult to justify rigorously. On the other hand,
the eigenfunctions can be factorised for special values of their arguments, leading to explicit
formulas which can be used to provide further evidence for the correctness of the expansions.
For the left eigenfunctions, these factorisations follow from a~result of Stade~\cite{stade-ajm} on Mellin
transforms of class one Whittaker functions and for the right eigenfunctions they follow from
an analogous result for certain coefficients of fundamental Whittaker functions
obtained in the present paper (see Theorem~\ref{ff}).

Assuming the expansions are correct, they reveal a~remarkable connection with
maximal heights of non-intersecting reflected Brownian bridges and non-intersecting
Brownian excursions. Recall that a~Brownian bridge, starting and ending at zero, is a~standard
one-dimensional Brownian motion, started at zero, and conditioned to be at zero at time 1.
A~reflected Brownian bridge, starting and ending at zero, is simply the absolute value of a~Brownian bridge, starting and ending at zero. A~(standard) Brownian excursion is a~Brownian bridge, starting and ending at zero, conditioned to stay positive in between.
A~collection of non-intersecting Brownian bridges (resp.\ reflected Brownian bridges or Brownian excursions)
is a~collection of independent Brownian bridges (resp.\ reflected Brownian bridges or Brownian excursions)
conditioned not to intersect.

As remarked above, \smash{$T_1\ed S$} where $S$ is defined by~\eqref{S}
and $\ed$ denotes equality in distribution.
In the case $r=2$, it was shown in~\cite[Proposition 4.4]{noc23} that
\smash{$T_2\ed S+S'$}, where $S'$ is an independent copy of $S$.
From these representations, one can infer, for example, the distribution functions\looseness=-1
\begin{equation}\label{df12} \P(T_1\le t) = \sum_{k\in\Z} (-1)^k {\rm e}^{-k^2 t}, \qquad
\P(T_2\le t) = \sum_{k\in\Z} \big(1-2k^2t\big) {\rm e}^{-k^2 t}.\end{equation}
The random variables $S$ and $S+S'$ have many interesting properties and interpretations~\cite{bpy}.
In particular, $S$ has the same law as twice the square of the maximum
of a~reflected Brownian bridge, starting and ending at zero, and $S+S'$ has the same law as twice
the square of the maximum of a~standard Brownian excursion.

Using our conjectured formulas, we carry out explicit computations for small values of $r$.
For~${r=1,2}$, we recover the formulas~\eqref{df12}.
For $r=3$, we obtain
\[
\P(T_3\le t) = \sum_{k,l\in\Z} P_{k,l}(t) {\rm e}^{-(k^2+l^2)t},
\]
where
\[P_{k,l}(t) = (-1)^{k+l} \big[1 - 2\big(k^2+l^2\big) t + \big(k^2-l^2\big)^2 t^2\big].\]
Using a~result from~\cite{lie},
we find that this agrees with the distribution function of
twice the square of the maximal height of a~pair of non-intersecting reflected Brownian bridges,
each starting and ending at zero.
The result we use from~\cite{lie} is the formula~(1.18) in that paper and to write it in the form required here we follow a~procedure similar to that given in~\cite{kat} for the absorbing case and then symmetrise the coefficients in the resulting
series~-- see Section~\ref{nibb} for more details.

For $r=4$, we obtain (partially numerically)
\[
\P(T_4\le t) = \sum_{k,l\in\Z} Q_{k,l}(t) {\rm e}^{-(k^2+l^2)t},
\]
where
\[
Q_{k,l}(t) = 1 - 4\big(k^2+l^2\big) t + 3\big(k^2+l^2\big)^2t^2
 -\frac23\big(k^2+l^2\big)^3t^3+\frac43 k^2l^2\big(k^2-l^2\big)^2 t^4.\]
This agrees with the distribution function of twice
the square of the maximal height of a~pair of non-intersecting Brownian excursions,
as computed in~\cite{f,kat,smcrf}.

This leads us to make the following conjecture. For $N\ge 1$, let $M_N$ (resp.\ $H_N$)
be the maximal height of $N$ non-intersecting reflected Brownian bridges starting and ending
at zero (resp.\ non-intersecting Brownian excursions).

\begin{conj} For each $N\ge 1$, \smash{$T_{2N-1}\ed 2(M_N)^2$}
and \smash{$T_{2N}\ed 2(H_N)^2$}.
\end{conj}

By the asymptotic results obtained in~\cite{lie} for the distributions of $M_N$
and $H_N$, this conjecture, if true, would imply the following:

\begin{conj}\label{conjecture1.2} As $r\to\infty$,
\[\P\big(T_r-2r\le xr^{1/3}\big) \to F_1 (x),
\]
where $F_1 (x)$ denotes the Tracy--Widom GOE distribution.
\end{conj}

Note that these conjectures may be interpreted as statements about the law of the interface between $0$'s and $1$'s, since the event $\{T_r\le t\}$ may be interpreted as saying that this interface has completely passed the diagonal $(1,r),(2,r-1),\dots, (r,1)$ by time $t$. Equivalently, if we define~${L_t=\max\{r \mid \pi_{ij}(t)=0}$ for 2$\le i+j\le r+1\}$ then $\{T_r\le t\}=\{L_t\ge r\}$. So, for example, Conjecture~\ref{conjecture1.2} is equivalent to the statement that, as $t\to\infty$,
\[\P\big(t/2-L_t \le 2^{-4/3} xt^{1/3}\big) \to F_1 (x).\]
As such, this strongly suggests that the interface itself, appropriately rotated, centered and rescaled, should converge in distribution to the Airy${}_2$ process.

Finally, we remark that, if we start the chain with $\pi_{ij}=1$ for all $i,j\ge 1$ then it is equivalent to the corner growth process and the absorption time $T_r$ may be interpreted as the point-to-line directed last passage time from $(1,1)$ to the line $(1,r),(2,r-1),\dots, (r,1)$ in a~model with standard exponential weights. It is known via (a continuous version of) the RSK correspondence~\cite{br,fw,f2010} that this random variable has the same law as the largest eigenvalue in a~certain Laguerre orthogonal ensemble and in~\cite{nr17} it was shown to have the same law as twice the square of the maximal height of $r$ non-intersecting Brownian bridges.

The outline of the paper is as follows. In the next section, we recall some
of the relevant background on the quantum Toda chain and its eigenfunctions,
which are known as Whittaker functions. We discuss in particular the discrete
quantisation of interest, and its eigenfunctions. The main new results in this
section are a~joint eigenfunction property of the right eigenfunctions (see Theorem~\ref{jef-f})
and a~factorisation for special values of their arguments (see Theorem~\ref{ff}).
As an aside, in Section~\ref{bs} we explain how this factorisation
yields a~series of binomial sum identities which generalise the classical identities of
Vandermonde and Chu. In Section~\ref{nibb}, we summarise
the relevant facts about maximal heights of non-intersecting Brownian bridges.
In Section~\ref{tpat}, we formulate and discuss our proposed
expansions for the discrete Whittaker process transition probabilities and absorption time distribution
functions, with some explicit computations for~${r\le 4}$ in support of the above conjectures.

\section{Eigenfunctions and their factorisations}\label{ef}

\subsection{The Toda chain}

Let
\begin{equation}\label{lax}
L=\begin{bmatrix} p_1 & -1 & 0 & \cdots & 0\\
q_1 & p_2 & -1 & \cdots & 0\\
0 & q_2 & \smash{\ddots} & \smash{\ddots} & \smash{\vdots}\\
\smash{\vdots}&&&p_r&-1\\
0&\dots&0&q_r &p_{r+1} \end{bmatrix},
\end{equation}
and define polynomials $\eta_{r,l}(p,q),\ 0\le l\le r+1$ by
\[\det(\lambda-L) = \sum_{l=0}^{r+1} (-1)^l \lambda^{r-l+1} \eta_{r,l}(p,q).\]
The determinant $\delta_{r}=\det(\lambda-L)$ may be computed using the recursion
\[\delta_l=(\lambda-p_{l+1}) \delta_{l-1} + q_{l} \delta_{l-2},\qquad l=0,\dots,r\]
with $\delta_{-1}=1$ and $\delta_{-2}=0$. The first three polynomials are given by
\[\eta_{r,0}=1,\qquad
\eta_{r,1}=\sum_{i=1}^{r+1} p_i, \qquad
\eta_{r,2}=\sum_{1\le i<j \le r+1} p_i p_j + \sum_{i=1}^r q_i.\]
More generally,
\begin{equation}\label{cdo}
\eta_{r,l} = \sum p_{i_1}\cdots p_{i_s} q_{j_1}\cdots q_{j_d}, \qquad 1\le l\le r+1, \end{equation}
where the sum is over $1\le i_1 < \cdots < i_s \le r+1$
and $1\le j_1< j_2<\cdots < j_d\le r$ such that $s+2d=l$,
$j_{k+1}>j_k+1$ for all $k$ and
$i_a \notin \{j_b,j_b+1\}$ for all $a$, $b$.

For example,
\begin{gather*}
\eta_{2,3}=p_1 p_2 p_3 + p_1 q_2 + q_1 p_3,\\
\eta_{3,3} = p_1p_2p_3+p_1p_2p_4+p_1p_3p_4+p_2p_3p_4
 + p_1q_2+p_1q_3+p_2q_3+q_1p_3+q_1p_4+q_2p_4,\\
\eta_{3,4}=p_1p_2p_3p_4+p_1p_2q_3+p_1q_2p_4+q_1p_3p_4+q_1q_3.
\end{gather*}

The Hamiltonians of the quantum Toda chain are the (commuting) differential operators
defined by $\H^{r,l}=\eta_{r,l}(p,q)$ with
\begin{equation}\label{pq-def}
p_i=\partial/\partial x_i,\qquad 1\le i \le r+1,\qquad q_i={\rm e}^{x_{i}-x_{i+1}},\qquad 1\le i\le r.
\end{equation}
This choice of $p$ and $q$ satisfy the commutation relations
\begin{equation}\label{cr}
[p_i,p_j]=[q_i,q_j]=0,\qquad [p_i,q_j]=q_i\delta_{ij}-q_j\delta_{i,j+1}.\end{equation}
The operators $\H^{r,l}$ are unambiguously defined
by the polynomials~\eqref{cdo}, as each term in the sum
only contains variables which commute with each other.

Denote by $L_k$ the partial backward shift operator defined by $L_k f(k) = f(k-1)$,
so, for example, if $f(n)$ is a~function of $n\in\Z_+^r$ (or $\C^r$), then
\[L_{n_i}f(n)=f(n-e_i)=f(n_1,\dots,n_{i-1},n_i-1,n_{i+1},\dots,n_r).\]
If we set $q_i=L_{n_i}$ and $p_i=n_i-n_{i-1}$, with the conventions
$n_0=n_{r+1}=0$, then these also satisfy the commutation relations~\eqref{cr}
and thus give rise to a~family of commuting difference operators~${h^{r,l}=\eta_{r,l}(p,q)}$, $0\le l\le r+1$.
We note that $h^{r,0}=1$, $h^{r,1}=0$ and $h^{r,2}=h^r$, where
\[h^r =\sum_{i=1}^r L_{n_i} - \sum_{i=1}^r n_i^2 + \sum_{i=1}^{r-1} n_i n_{i+1}. \]
Note that the recursion~\eqref{rec-ar} may be written as $h^r a_r=0$.

Denote by $~^*\! \H^{r,l}$ and $~^*\! h^{r,l}$
the formal adjoints of these operators. More precisely, we define
$~^*\! \H^{r,l}=\eta_{r,l}(-p,q)$ with $p$, $q$ given by~\eqref{pq-def} and
$~^*\! h^{r,l} = \eta_{r,l}(p,q^*)$ with $p_i=n_i-n_{i-1}$
(with the conventions $n_0=n_{r+1}=0$) and $q_i^*=R_{n_i}$,
where $R_k$ denotes the partial forward shift operator defined by $R_k f(k) = f(k+1)$,
so, for example, if $f(n)$ is a~function of $n\in\Z_+^r$ (or $\C^r$) then
\[R_{n_i}f(n)=f(n+e_i)=f(n_1,\dots,n_{i-1},n_i+1,n_{i+1},\dots,n_r).\]

These two different quantisations are related via the duality relation
\begin{equation}\label{hH-rel}
h_s^{r,l} F(s,x) = {}^*\! \H_x^{r,l} F(s,x),
\end{equation}
where
\[F(s,x)=y^{-s}\equiv \prod_{i=1}^r y_i^{-s_i}=\prod_{i=1}^r {\rm e}^{s_i(x_{i+1}-x_{i})}.\]
This follows immediately from the relations
\[L_{s_i} F(s,x) = y_i F(s,x),\qquad (s_i-s_{i-1}) F(s,x) = - \partial_{x_i} F(s,x),\]
with the convention $s_0=s_{r+1}=0$.

\begin{rem}\label{rem-H}
If we set $q_i=n_i^2 L_{n_i}$ and $p_i=n_i-n_{i-1}$, again with the conventions
$n_0=n_{r+1}=0$, then these also satisfy~\eqref{cr}
and thus give rise to a~family of commuting difference operators
$H^{r,l}=\eta_{r,l}(p,q)$, $0\le l\le r+1$. In this case $H^{r,0}=1$, $H^{r,1}=0$ and
$H^{r,2}=H^r$, where
\[H^r = \sum_{i=1}^r n_i^2 D_{n_i} +\sum_{i=1}^{r-1} n_i n_{i+1}\]
and $D_k = L_k-I$ is the partial backward difference operator defined by $D_k f(k) = f(k-1) - f(k)$,
so, for example, if $f(n)$ is a~function of $n\in\Z_+^r$ (or $\C^r$) then
\[D_{n_i}f(n)=f(n-e_i)-f(n).\]

The difference operators $h^r$ and $H^r$ are related by
\[h^r=N_r(n)^{-1}\circ H^r \circ N_r(n),\qquad N_r(n)=\prod_{i=1}^r n_i!^2.\]
If we define
\begin{equation}\label{Ar-def}
A_r(n)=N_r(n) a_r(n),
\end{equation} then $A_r$ is the unique solution to $H^rA_r=0$
with $A_r(0)=1$ and the convention $A_r(n)=0$ for~${n\notin\Z_+^r}$.
\end{rem}

\subsection{Class one Whittaker functions}

For $\mu\in\C$, $x\in\R^{r+1}$ and $x'\in\R^r$, define
\[
\Q^{(r)}_\mu(x,x')=\exp\left( \mu \left(\sum_{i=1}^{r+1}x_i-\sum_{i=1}^r x'_i\right)
-\sum_{i=1}^r\big({\rm e}^{x_i-x'_i}+{\rm e}^{x'_i-x_{i+1}}\big)\right).\]
Set \smash{$\psi_\nu^{(0)}(x)={\rm e}^{\nu x}$} and, for $r\ge 1$, $\nu\in\C^{r+1}$ and $x\in\R^{r+1}$,
\[
\psi^{(r)}_{\nu_1,\dots,\nu_{r+1}}(x)=
\int_{\R^r} \Q^{(r)}_{\nu_{r+1}}(x,x') \psi^{(r-1)}_{\nu_1,\dots,\nu_{r}}(x'){\rm d}x'.
\]
Then $\psi_\nu=\psi^{(r)}_\nu$ satisfies the eigenvalue equations
\begin{equation}\label{jep} \H^{r,l} \psi_\nu = e_l(\nu) \psi_\nu,\qquad 1\le l\le r+1,\end{equation}
where $e_l(\nu)$ denotes the $l^{th}$ elementary symmetric polynomial
\[
e_l(\nu)=\sum_{i_1<\cdots<i_l} \nu_{i_1}\dots\nu_{i_l}.
\]
Equivalently,
$\Delta^{(r)}(\lambda) \psi_\nu = c_\lambda(\nu) \psi_\nu$,
where
\[\Delta^{(r)}(\lambda) = \sum_{l=0}^{r+1} (-1)^l \lambda^{r-l+1} \H^{r,l},\qquad c_\lambda(\nu)=\prod_{i=1}^{r+1} (\lambda-\nu_i).\]

These eigenfunctions are called class one $\operatorname{GL}(r+1,\R)$ Whittaker functions,
for more background see~\cite{gklo} and references therein.
The joint eigenfunction property~\eqref{jep} may be seen as a~consequence
of the intertwining relation~\cite[Lemma 4.1]{gklo}
\[
\Delta^{(r)}(\lambda) \circ \Q^{(r)}_\mu = (\lambda-\mu) \Q^{(r)}_\mu \circ \Delta^{(r-1)}(\lambda).
\]
Although it is not obvious from the above definition,
the function $\psi_\nu(x)$ is symmetric in the parameters $\nu_i$.
If $r=1$ and $\nu_1+\nu_2=0$, then
$\psi_\nu(x)=2 K_{\nu_1-\nu_2}\bigl( 2 {\rm e}^{(x_1-x_2)/2}\bigr)$,
where $K_a$ is the modified Bessel function of the second kind
\[K_a(z)= \frac12 \int_0^\infty t^{a-1} \exp\left(-\frac{z}{2}\bigl(t+t^{-1}\bigr)\right) {\rm d}t.\]

Let us write \smash{$\C_0^{r+1}=\bigl\{\nu\in\C^{r+1}\mid \sum_i \nu_i=0\bigr\}$}.
If $\nu\in\C_0^{r+1}$, then $\psi_\nu(x)$ only depends on $x$
through the variables $y_i={\rm e}^{x_i-x_{i+1}}$, $i=1,\dots,r$, and
so in this case let us denote $\Psi_\nu(y)=\psi_\nu(x)$.

\subsection{Fundamental Whittaker functions}\label{fwf}

Let $\nu\in\C_0^{r+1}$ and consider the series
\[
\phi_\nu(x)=\sum_n a_{r,\nu}(n) {\rm e}^{(\xi_n+\nu,x)},
\]
where the sum is over $n\in \Z_+^{r}$ and
$\xi_n=\sum_{i=1}^{r} n_i (e_i-e_{i+1})$.
Then $\phi_\nu$ satisfies the eigenvalue equation
$\H^{r,2}\phi_\nu = e_2(\nu)\phi_\nu$,
at least formally, provided the coefficients $a_{r,\nu}(n)$ satisfy the recursion\looseness=-1%
\begin{equation}\label{rec} \left[\sum_{i=1}^{r} n_i^2 - \sum_{i=1}^{r-1}n_i n_{i+1}
+\sum_{i=1}^{r} (\nu_i-\nu_{i+1}) n_i \right]
a_{r,\nu}(n) = \sum_{i=1}^{r} a_{r,\nu}(n-e_i),\end{equation}
with the convention $a_{r,\nu}(n) =0$ for $n\notin \Z_+^{r}$.
Since $\sum_i\nu_i=0$, the function $\phi_{\nu}(x)$ only depends on~$x$
through $y_i={\rm e}^{x_i-x_{i+1}}$, $i=1,\dots,r$, so let us denote $\Phi_\nu(y)=\phi_\nu(x)$.
These functions, originally introduced by Hashizume~\cite{h},
are known as fundamental $\operatorname{GL}(r+1,\R)$ Whittaker functions.

Ishii and Stade~\cite{is} obtained the following recursive formula
for the coefficients $a_{r,\nu}(n)$.
For~${\nu\in\C_0^{r+1}}$, $n\in \Z_+^{r}$ and $k\in \Z_+^{r-1}$, define
\[
q_{r,\nu}(n,k)=\prod_{i=1}^{r} \frac1{(n_i-k_i)! \Gamma(n_i-k_{i-1}+\nu_i-\nu_{r+1}+1)},
\]
with the convention $k_0=k_{r}=0$.
Define
$a_{r,\nu}(n)$, $r\ge 1$, $n\in \Z_+^{r}$, $\nu\in\C_0^{r+1}$ recursively by
\begin{equation}\label{def-a1} a_{1,\nu}(n) =\frac1{n! \Gamma(n+\nu_1-\nu_2+1)}\end{equation}
and, for $r\ge 2$,
\begin{equation}\label{def-ar}
a_{r,\nu}(n) =\sum_{k} q_{r,\nu}(n,k) a_{r-1,\mu}(k),
\end{equation}
where $\mu_i=\nu_i+\nu_{r+1}/r$, $1\le i\le r$, and the sum is over $k\in \Z_+^{r-1}$.
Then, for each $r\ge 1$,
$a_{r,\nu}(n)$ satisfies the recursion~\eqref{rec} with
\[ a_{r,\nu}(0) =\prod_{i<j}\frac{1}{\Gamma(\nu_i-\nu_j+1)}.\]
It also follows from the definition that, for $k\in\Z_+^r$,
\begin{equation}\label{ar0}
a_{r,\nu}(0,k)=\prod_{j=2}^{r+1} \frac{1}{\Gamma(\nu_1-\nu_j+1)} a_{r-1,\tau}(k),
\end{equation}
where $\tau_i=\nu_{i+1}+\nu_1/r$, $i=1,\dots,r$.

If $r=1$, then
$\Phi_\nu(y)=I_a\left(2\sqrt{y}\right)$,
where $a=\nu_1-\nu_2$ and $I_a$ is the modified Bessel function of the first kind
\[I_a(z)=\sum_{n=0}^\infty \frac1{n! \Gamma(n+a+1)} \left(\frac{z}{2}\right)^{2n+a}.\]
For $r=2$, we have
\begin{equation}\label{bf}
a_{2,\nu}(n,m)=\frac{\Gamma(n+m+a+b+1)}{n! \Gamma(n+a+1) \Gamma(n+a+b+1) m! \Gamma(m+b+1) \Gamma(m+a+b+1)},
\end{equation}
where $a=\nu_1-\nu_2$, $b=\nu_2-\nu_3$. The formula~\eqref{bf} is due to Bump~\cite{bump}.

When $\nu=0$, the recursion~\eqref{def-ar} may be conveniently expressed as follows.
Let $\Pi^r$ be the set of arrays $(\pi_{ij},\ 2\le i+j\le r+1)$ of non-negative integers
satisfying $ \pi_{ij} \ge \max\{\pi_{i,j-1},\pi_{i-1,j}\}$,
with the convention $\pi_{i0}=\pi_{0j}=0$. For $\pi\in\Pi^r$, let $b(\pi)\in\Z_+^r$
denote it `boundary values' $b(\pi)=(\pi_{1,r},\pi_{2,r-1},\dots,\pi_{r,1})$ and
let $\Pi^r_n$ be the set of $\pi\in\Pi^r$ with $b(\pi)=n$. For $\pi\in\Pi^r$, set
\[ w(\pi)=\prod_{1\le i+j \le r+1} \frac1{(\pi_{ij}-\pi_{i,j-1})! (\pi_{ij}-\pi_{i-1,j})!}.\]
Then~\eqref{def-ar} may be written as
\[ a_r(n) = \sum_{\pi\in\Pi^r_n} w(\pi). \]
Equivalently, the normalised coefficients $A_r(n)$ defined by~\eqref{Ar-def}
are given by
\begin{equation}\label{Ar-bs} A_r(n) = \sum_{\pi\in\Pi^r_n} \prod_{1\le i+j \le r+1} \binom{\pi_{i,j}}{\pi_{i-1,j}} \binom{\pi_{i,j}}{\pi_{i,j-1}}. \end{equation}

\subsection{Mellin transforms of class one Whittaker functions}\label{mellin}

For $\nu\in\C_0^{r+1}$ and $s\in\C^r$, with $\operatorname{Re}(s_i)$ sufficiently large, define
\[
M_{r,\nu}(s)=\int_{\R^r} \Psi_\nu(y) \prod_{i=1}^r y_i^{s_i} \frac{{\rm d}y_i}{y_i}.
\]
From the corresponding property of $\Psi_\nu$,
the Mellin transform $M_{r,\nu}(s)$ is symmetric in the parameters $\nu_1,\dots,\nu_{r+1}$
and extends to a~meromorphic function of $s\in\C^r$~\cite{fg}.
By~\eqref{hH-rel} and~\eqref{jep}, this extension satisfies the eigenvalue equations
\begin{equation}\label{jepm} h_s^{r,l} M_{r,\nu}(-s) = e_l(\nu) M_{r,\nu}(-s),\qquad 2\le l\le r+1,\end{equation}
or, equivalently,
\begin{equation}\label{msum}
d_s^{(r)}(\lambda) M_{r,\nu}(-s) = c_\lambda(\nu) M_{r,\nu}(-s),
\end{equation}
where
\[
d^{(r)}(\lambda) = \sum_{l=0}^{r+1} (-1)^l \lambda^{r-l+1} h^{r,l},\qquad c_\lambda(\nu)
=\prod_{i=1}^{r+1} (\lambda-\nu_i).
\]

We remark that the difference equations~\eqref{jepm} are in fact equivalent to those
recently obtained by Stade and Trinh~\cite{st}. Indeed, the main result
of~\cite{st} may be stated, in the present notation, as the statement that
\smash{$d_s^{(r)}(\nu_i) M_{r,\nu}(-s)=0$} for $i=1,\dots,r+1$. On the other hand,
\smash{$d_s^{(r)}(\lambda) M_{r,\nu}(-s)$} is a~polynomial in $\lambda$ of degree $r+1$, so
must agree with $c_\lambda(\nu)$ up to a~factor which does not depend on $\lambda$.
Letting $\lambda\to\infty$ and recalling that $h^{r,0}=1$, we recover the identity~\eqref{msum}.

For $r=1,2$, we have the explicit formulas~\cite{bump,is}
\[M_{1,\nu} (s)=\Gamma(s+\nu_1)\Gamma(s+\nu_2),\qquad
M_{2,\nu}(s)=\Gamma(s_1+s_2)^{-1} \prod_{i=1}^3 \Gamma(s_1+\nu_i) \Gamma(s_2-\nu_i).
\]
The following remarkable factorisation in the general case was
originally conjectured by Bump and Friedberg~\cite{bf}
and later proved by Stade~\cite{stade-ajm}.
\begin{thm}[Stade]\label{stade}
Let $s_1,s_2\in\C$ and, for $i\ge 3$,
\[s_i=\begin{cases} is_2/2, & i \mbox{ \rm even},\\ s_1+(i-1)s_2/2, & i \mbox{ \rm odd},\end{cases}\]
so that
\[s_1,s_2,s_3,\ldots = s_1,s_2,s_1+s_2,2s_2,s_1+2s_2,3s_2,s_1+3s_2,\ldots.\]
Then
\begin{equation}\label{mf}
\Gamma(s_{r+1}) M_{r,\nu}(s_1,\dots,s_r)=
\prod_{i=1}^{r+1} \Gamma(s_1+\nu_i) \prod_{1\le i<j\le r+1} \Gamma(s_2+\nu_i+\nu_j).\end{equation}
\end{thm}

\subsection{Shifted coefficients of fundamental Whittaker functions}\label{fun}

For $\nu\in\C^{r+1}_0$, let us denote $\nu'\in\C^r$, $\nu'_i=\sum_{j=1}^i\nu_j$ for $i=1,\dots,r$.
For $\nu\in\C^{r+1}_0$ and $n\in\nu'+\Z_+^r$, define
\begin{equation}\label{fa} f_{r,\nu}(n)= a_{r,\nu}(n-\nu'), \end{equation}
with the convention $f_{r,\nu}(n)=0$ if $n\in ( \nu'+\Z^r)\backslash(\nu'+\Z_+^r)$.
For example,
\[f_{1,\nu}(n)=\frac1{\Gamma(n-\nu_1+1)\Gamma(n-\nu_2+1)},\]
and, by~\eqref{bf},
\[ f_{2,\nu}(n,m)=\frac{\Gamma(n+m+1)}{\prod_{i=1}^3 \Gamma(n-\nu_i+1)\Gamma(m+\nu_i+1)}.\]

In this notation, the recursion~\eqref{rec} may be written as $h^r f_{r,\nu} = e_2(\nu) f_{r,\nu}$.
In this section, we will show that
$h^{r,l} f_{r,\nu} = e_l(\nu) f_{r,\nu}$, $ 2\le l\le r+1$.
We will also use this to obtain a~useful factorisation of $f_{r,\nu}(n)$, analogous to~\eqref{mf},
for special values of the argument.

The recursive definition~\eqref{def-a1} and~\eqref{def-ar} may be reformulated to give a~recursive
definition of the functions $f_{r,\nu}$, as follows.
For $\theta\in\C$, $n\in \C^{r}$ and $k\in \C^{r-1}$, define
\[
\Lambda^\theta_{r}(n,k)=\prod_{i=1}^{r} \frac1{\Gamma(n_i-k_i+i\theta/r+1) \Gamma(n_i-k_{i-1}-(r-i+1)\theta/r+1)},
\]
with the convention $k_0=k_{r}=0$.

Define $f_{r,\nu}(n)$, $r\ge 1$, $\nu\in\C_0^{r+1}$, $n\in \nu'+\Z_+^{r}$, recursively by
\[ f_{1,\nu}(n) =\frac1{\Gamma(n-\nu_1+1) \Gamma(n+\nu_1+1)}\]
and, for $r\ge 2$,
\begin{equation}\label{def-fr}
f_{r,\nu}(n) =\sum_{k} \Lambda^{\nu_{r+1}}_{r}(n,k) f_{r-1,\mu}(k),
\end{equation}
where $\mu_i=\nu_i+\nu_{r+1}/r$, $1\le i\le r$, and the sum is over $k\in\mu'+ \Z_+^{r-1}$.

Recall that $h^{r,l}_n =\eta_{r,l}(p,q)$ and $~^*\! h^{r,l}_n =\eta_{r,l}(p,q^*)$
with $q_i=L_{n_i}$, $q^*_i=R_{n_i}$ and $p_i=n_i-n_{i-1}$, with the
conventions $n_0=n_{r+1}=0$. Let
\begin{gather*}
d^{(r)}(\lambda)=\sum_{l=0}^{r+1} (-1)^l \lambda^{r-l+1} h^{r,l},\qquad
 {}^*\! d^{(r)}(\lambda)=\sum_{l=0}^{r+1} (-1)^l \lambda^{r-l+1} {}^*\! h^{(r,l)},\\
c_\lambda(\nu)=\prod_{i=1}^{r+1} (\lambda-\nu_i).
\end{gather*}

\begin{thm}\label{jef-f}
The functions $f_{r,\nu}(n)$ satisfy
\begin{equation}
h^{r,l} f_{r,\nu} = e_l(\nu) f_{r,\nu}, \qquad 2\le l\le r+1.
\end{equation}
Equivalently,
\begin{equation}\label{jef-f2} d^{(r)}_n(\lambda) f_{r,\nu}= c_\lambda(\nu) f_{r,\nu}.\end{equation}
 \end{thm}
\begin{proof}
We will show that
\begin{equation}\label{int-delta} d^{(r)}_n(\lambda) \Lambda_r^\theta(n,k)
= (\lambda-\theta) {}^*\! d^{(r-1)}_k(\lambda+\theta/r) \Lambda_r^\theta(n,k).\end{equation}
To see that this implies the statement of the proposition, note that,
from the definition, $f_{r-1,\mu} (k)$ vanishes whenever $k_i=\mu_i'-1$, for some $i$.
The identity~\eqref{jef-f2} therefore follows from~\eqref{int-delta}.

Dividing both sides of~\eqref{int-delta} by $\Lambda_r^\theta(n,k)$, it reduces to
$\det (\lambda-\theta-A) = \det(\lambda-\theta-B)$,
where \smash{$A=(a_{ij})_{i,j=1}^{r+1}$} and \smash{$B=(b_{ij})_{i,j=1}^{r+1}$} are the tridiagonal matrices defined by
\begin{gather*}
a_{ij} = \begin{cases} -1, & j=i+1,\\
n_j-n_{j-1}-\theta ,& j=i,\\
\left(n_j-k_j+\dfrac{j\theta}{r}\right)\left(n_j-k_{j-1}-\dfrac{(r-j+1)\theta}{r}\right), & j=i-1,\\
0 ,& \mbox{otherwise}, \end{cases}
\\
b_{ij} = \begin{cases} -1, & j=i+1,\\
k_j-k_{j-1}-(r+1)\theta/r ,& j=i \mbox{ and } i\le r,\\
\left(n_j-k_j+\dfrac{j\theta}{r}\right)\left(n_i-k_{i-1}-\dfrac{(r-i+1)\theta}{r}\right), & j=i-1 \mbox{ and } i\le r,\\
0 ,& \mbox{otherwise}, \end{cases}
\end{gather*}
with the usual conventions $n_0=n_{r+1}=0$ and $k_0=k_r=0$.
The result follows, noting that $A=CD$ and $B=DC$, where
\begin{gather*}
C=\begin{bmatrix} 1 & 0 & 0 & \cdots & 0\\
n_1-k_1 +\frac{\theta}{r}& 1 & 0 & \cdots & 0\\
0 & n_2-k_2 +\frac{2\theta}{r} & \smash{\ddots} & \smash{\ddots} & \smash{\vdots}\\
\smash{\vdots}&&&1&0\\
0&\dots&0&n_r + \theta &1 \end{bmatrix},\\
D=\begin{bmatrix} n_1 -\theta & -1 & 0 & \cdots & 0\\
0 & n_2-k_1 -\frac{(r-1)\theta}{r}& -1 & \cdots & 0\\
0 & 0 & \smash{\ddots} & \smash{\ddots} & \smash{\vdots}\\
\smash{\vdots}&&&n_r-k_{r-1}-\frac{\theta}{r}&-1\\
0&\dots&0&0 &0 \end{bmatrix}.\tag*{\qed}
\end{gather*}\renewcommand{\qed}{}
\end{proof}

\begin{rem}
In the case $\nu=0$, Theorem~\ref{jef-f} states that $a_r(n)$ is annihilated by
the commuting difference operators $h^{r,l}$, $2\le l\le r+1$ or, equivalently,
that $A_r(n)$ is
annihilated by the commuting difference operators $H^{r,l}$, $2\le l\le r+1$,
as defined in Remark~\ref{rem-H}.
For example,
\[A_2(n,m)=\sum_k {n \choose k}{m \choose k} ={n+m\choose n}\]
is annihilated by the two commuting difference operators
\[H^{2,2}=n^2 D_n+m^2 D_m+nm,\qquad H^{2,3}=-n^2 mD_n+nm^2 D_m.\]
We remark that
\[D_n+D_m+1 = \frac{1}{nm} H^{2,2}+\frac{n-m}{n^2m^2} H^{2,3},\]
and so we recover Pascal's identity
$A_2(n,m)=A_2(n-1,m)+A_2(n,m-1)$.
For $r=3$, Theorem~\ref{jef-f} states that
\[A_3(n,m,l) = \sum_k {m \choose k}^2 {n+k \choose m} {l+k \choose m}\]
is annihilated by the three commuting difference operators
\begin{align*}
H^{3,2} &=n^2 D_n + m^2 D_m +l^2 D_l+nm+ml,\\
H^{3,3} &=-mn^2D_n+(n-l)m^2D_m+ml^2D_l,\\
H^{3,4} &=n^2l^2D_{n,l}-l(l-m)n^2D_n-nm^2lD_m+n(m-n)l^2D_l.
\end{align*}
Here we are using the notation $D_{n,l}f(n,l)=f(n-1,l-1)-f(n,l)$.
As remarked in \cite[Remark~2.5]{noc23},
the diagonal values $a_n=A_3(n,n,n)$ are the Ap\'ery numbers
\[a_n=\sum_k {n \choose k}^2 {n +k \choose k}^2\]
associated with $\zeta(3)$. It is well known that this sequence satisfies the three-term recurrence
\[
n^3 a_n = \big(34n^3-51n^2+27n-5\big)a_{n-1}-(n-1)^3 a_{n-2},
\]
with $a_0=1$ and $a_1=5$. We remark that this recurrence may
also be derived using the difference equations
$H^{3,2}A_3=H^{3,3}A_3=H^{3,4}A_3=0$.
\end{rem}

\begin{thm}\label{ff} Let $n,m\in\C$ and define, for $i\ge 1$,
\[ n_i=\begin{cases}
 im/2 ,& i \mbox{ \rm even},\\ n+(i-1)m/2, & i \mbox{ \rm odd},\end{cases}\]
so that
\begin{equation}\label{n1n2} n_1,n_2,n_3,\ldots =n,m,n+m,2m,n+2m,3m,n+3m,4m,\ldots . \end{equation}
Suppose $\nu\in\C_0^{r+1}$, $(n_1,\dots,n_r)\in \nu'+\Z_+^{r}$ and $\operatorname{Re}(n_{r+1})>-1$. Then
\[f_{r,\nu}(n_1,\dots,n_r) = \Gamma(n_{r+1}+1) \prod_i \frac1{\Gamma(n-\nu_i+1) }\prod_{i<j} \frac1{\Gamma(m-\nu_i-\nu_j+1)}.\]
\end{thm}
\begin{proof}
We prove this by induction over $r$. For $r=1$, it follows from the definition.

Let us define, for $i\ge 1$,
\[ p_i(n,m)=\begin{cases} im/2 ,& i \mbox{ \rm even},\\ n+(i-1)m/2 ,& i \mbox{ \rm odd},\end{cases}\]
and, for each $r\ge 1$,
$p^r(n,m)=(p_1(n,m),\dots,p_r(n,m))$.
Note that $n_i=p_i(n,m)$ and $(n_1,\dots,\allowbreak n_r)=p^r(n,m)$.
We also note the identity
\begin{equation}\label{pki}
p_{i+1}(n,m)-n+ix=p_i(m-n+x,m+2x).
\end{equation}

Let us denote
$g_{r,\nu}(n,m)=f_{r,\nu}(p^r(n,m))=f_{r,\nu}(n_1,\dots,n_r)$.
We first consider the case $n=\nu_1$. From the relation~\eqref{fa},
\[g_{r,\nu}(\nu_1,m) = f_{r,\nu}(p^r(\nu_1,m)) = a_{r,\nu}(p^r(\nu_1,m)-\nu') = a_{r,\nu}(0,k),\]
where $k_i = p_{i+1}(\nu_1,m)-\nu'_{i+1} \in\Z_+$, $i=1,\dots,r-1$. Thus, by~\eqref{ar0},
\[g_{r,\nu}(\nu_1,m) = \prod_{j=2}^{r+1} \frac{1}{\Gamma(\nu_1-\nu_j+1)} a_{r-1,\tau}(k),\]
where $\tau_i=\nu_{i+1}+\nu_1/r$, $i=1,\dots,r$. Using~\eqref{fa} again, we may write
$a_{r-1,\tau}(k) = f_{r-1,\tau}(k+\tau')$.
The identity~\eqref{pki} gives, for $i=1,\dots,r-1$,
\begin{align*}
 k_i+\tau'_i &= p_{i+1}(\nu_1,m)-\nu'_{i+1} +\tau'_i = p_{i+1}(\nu_1,m) - \nu_1 + i\nu_1/r \\
&= p_i\left(m-\nu_1+\frac{\nu_1}{r},m+\frac{2\nu_1}{r}\right)
\end{align*}
and, moreover, taking $i=r$,
\begin{equation}\label{pri} p_{r}\left(m-\nu_1+\frac{\nu_1}{r},m+\frac{2\nu_1}{r}\right) = p_{r+1}(\nu_1,m).\end{equation}
Hence
\[g_{r,\nu}(\nu_1,m) = \prod_{j=2}^{r+1} \frac{1}{\Gamma(\nu_1-\nu_j+1)} g_{r-1,\tau}\left(m-\nu_1+\frac{\nu_1}{r},m+\frac{2\nu_1}{r}\right).\]
By the induction hypothesis, and~\eqref{pri},
\begin{gather*}
 g_{r-1,\tau}\left(m-\nu_1+\frac{\nu_1}{r},m+\frac{2\nu_1}{r}\right)\\
\qquad= \Gamma\left(p_{r}\left(m-\nu_1+\frac{\nu_1}{r},m+\frac{2\nu_1}{r}\right)+1\right)\\
\phantom{\qquad= }{}\times
 \prod_{j=2}^{r+1} \frac1{\Gamma(m-\nu_1-\nu_j+1) }\prod_{2\le i<j\le r+1} \frac1{\Gamma(m-\nu_i-\nu_j+1)} \\
\qquad= \Gamma\left(p_{r+1}(\nu_1,m)+1\right) \prod_{1\le i<j\le r+1} \frac1{\Gamma(m-\nu_i-\nu_j+1)}.
\end{gather*}
Thus,
\[g_{r,\nu}(\nu_1,m) = \Gamma\left(p_{r+1}(\nu_1,m)+1\right) \prod_{j=2}^{r+1} \frac{1}{\Gamma(\nu_1-\nu_j+1)} \prod_{1\le i<j\le r+1} \frac1{\Gamma(m-\nu_i-\nu_j+1)},\]
as required.

It therefore remains to show that
\[
g_{r,\nu}(n,m) = \frac{\Gamma(p_{r+1}(n,m)+1)}{\Gamma(p_{r+1}(\nu_1,m)+1)}
\prod_{j=1}^{r+1} \frac{\Gamma(\nu_1-\nu_j+1)}{\Gamma(n-\nu_j+1)} g_{r,\nu}(\nu_1,m).
\]
We will deduce this from the following recursion:
\begin{equation}\label{rec-alpha}[h_r(n,m) L_n - c_n(\nu)] g_{r,\nu}(n,m) = 0,\end{equation}
where
\[h_r(n,m)=\begin{cases}
1 ,& r \mbox{ odd},\\
(n+rm/2), & r \mbox{ even}.\end{cases}\]

To prove~\eqref{rec-alpha}, we use the following identity: for
\[d^{(r)}_{n_1,\dots,n_r}(\lambda) = \sum_{l=0}^{r+1} (-1)^l \lambda^{r-l+1} h^{r,l}_{n_1,\dots,n_r}\]
with $n_1,n_2,\dots$ given by~\eqref{n1n2} and $\lambda=n$,
\begin{equation}
\label{id} d^{(r)}_{(n_1,\dots,n_r)}(\lambda)=
\begin{cases} q_1q_3 \dots q_r, & r \mbox{ odd},\\
(n+rm/2) q_1q_3 \dots q_{r-1} ,& r \mbox{ even}.\end{cases}\end{equation}
Indeed, let us denote by $L^{(r)}$ the Lax matrix~\eqref{lax} with
$p=p^{(r)}=(n_1,n_2-n_1,\dots,n_{r+1}-n_r,-n_{r+1})$
and $n_1,n_2,\dots$ given by~\eqref{n1n2}.
Then
\[\delta^{(r,r+1)} = \big(n-p^{(r)}_{r+1}\big) \delta^{(r,r)} + q_r \delta^{(r,r-1)},\]
where $\delta^{(r,j)}$ denotes the $j^{th}$ principal minor of $n-L^{(r)}$.
Moreover,
\[\delta^{(r,r)}=\begin{cases}
\delta^{(r-1,r)}, & r \mbox{ even},\\
0, & r \mbox{ odd},\end{cases}\]
and
\[\delta^{(r,r-1)}=\begin{cases}
\delta^{(r-2,r-1)}, & r \mbox{ odd},\\
0 ,& r \mbox{ even}.\end{cases}\]
Putting these together yields the identity~\eqref{id}.

Now, by Theorem~\ref{jef-f}, we have
\smash{$\big[d_{(n_1,\dots,n_r)}^{(r)}(\lambda)-c_\lambda(\nu)\big] f_{r,\nu}(n_1,\dots,n_r)=0$}.
The recursion~\eqref{rec-alpha} follows, taking
$n_1,n_2,\dots$ as in~\eqref{n1n2}, $\lambda=n$, $q_i= L_{n_i}$
and using~\eqref{id}.
\end{proof}

\subsection{Binomial sum identities}\label{bs}

In the case $\nu=0$, Theorem~\ref{ff} may be reformulated as a~series of
binomial sum identities.
Recall that the normalised coefficients $A_r(n)=N_r(n)a_r(n)$ are given
by the binomial sums~\eqref{Ar-bs}.
For small values of $r$, these sums may be simplified.
For $r=2$, by Vandermonde's identity,
\[A_2(n,m)=\sum_k {n \choose k}{m \choose k} = {n+m\choose n}.\]
For $r=3$, using
\[{m \choose i} {i\choose k} = {m\choose k} {m-k\choose m-i},\]
and Vandermonde's identity (twice), we obtain
\begin{align}\label{A3}
A_3(n,m,l) &= \sum_{i,j,k} {n \choose i}{m \choose i}{i\choose k}
{l \choose j} {m \choose j} {j\choose k} = \sum_{i,j,k} {m \choose k}^2 {n \choose i} {m-k\choose m-i} {l \choose j} {m-k\choose m-j} \nonumber\\
&=\sum_k {m \choose k}^2 {n+m-k \choose m} {m+l-k \choose m}=\sum_k {m \choose k}^2 {n+k \choose m} {l+k \choose m}.
\end{align}
We note that this is equivalent to
\[A_3(n,m,l)= {n+m\choose n} {m+l\choose l} {}_4 F_3\left[
\! \begin{array}{c} -n,-m,-m,-l \\ 1,-n-m,-m-l \end{array} \middle| 1\right],\]
where $~_p F_q$ denotes the generalised hypergeometric series
\[~_p F_q\left[
\! \begin{array}{c} a_1,\dots,a_p \\ b_1,\dots,b_q \end{array} \middle| z\right]=
\sum_{k=0}^\infty \frac{(a_1)_k\cdots (a_p)_k}{(b_1)_k\cdots (b_q)_k} \frac{z^k}{k!},\qquad (x)_n=\frac{\Gamma(x+n)}{\Gamma(x)}.\]

A~similar computation yields
\begin{gather*}\label{a4}
A_4(n_1,n_2,n_3,n_4)\nonumber \\
\qquad= \sum_{i,j}
{n_2 \choose i}^2 {n_1+n_2-i \choose n_2} {n_3 \choose j}^2 {n_3+n_4-j \choose n_3}
{n_2+n_3-i-j \choose n_2-j} {i+j \choose i}.
\end{gather*}

Let $n,m\in\Z_+$ and define, for $i\ge 1$,
\[
n_i=\begin{cases} im/2, & i \mbox{ \rm even},\\ n+(i-1)m/2, & i \mbox{ \rm odd},\end{cases}\]
so that
$n_1,n_2,n_3,\ldots =n,m,n+m,2m,n+2m,3m,n+3m,4m,\ldots$.
For $r\ge 2$, define \[F_r=F_r(n,m)=A_r(n_1,\dots,n_r).\]
Define, for $k\ge 1$,
\[G_{2k-1}={km \choose m,\dots,m}=\prod_{j=1}^k {jm \choose m},\qquad G_{2k}={n+km \choose n,m,\dots,m}={n+km \choose n} G_{2k-1}.\]
In this notation, Theorem~\ref{ff} yields the following.
\begin{cor}
$F_2=G_2$ and, for $r> 2$,
$F_r=G_{r} G_{r-1}^2 \cdots G_2^2$.
\end{cor}
For example,
\[F_2={n+m \choose n}, \qquad F_3={n+m \choose n}^2 {2m \choose m},\qquad
F_4={n+m \choose n}^2 {n+2m \choose n} {2m \choose m}^3,\]
\[F_5={n+m \choose n}^2 {n+2m \choose n}^2 {3m \choose m} {2m \choose m}^5.\]
The above formula for $F_2$ is an immediate consequence of
Vandermonde's identity. The expression for $F_3$ may be obtained directly
using~\eqref{A3} and Chu's identity
\[\sum_k {m \choose k}^2 {n+2m-k \choose 2m} = {n+m \choose n}^2.\]
Indeed,
\begin{align*}
F_3=A_3(n,m,n+m) &= \sum_k {m \choose k}^2 {n+m-k \choose m} {n+2m-k \choose m}\\
&= \sum_k {m \choose k}^2 {n+2m-k \choose 2m} {2m \choose m} = {n+m \choose n}^2 {2m \choose m}.
\end{align*}
By~\eqref{a4}, the above formula for $F_4$ is equivalent to the identity
\begin{gather*}
\sum_{i,j}
{m \choose i}^2 {n+m-i \choose m} {n+m \choose j}^2 {n+3m-j \choose n+m}
{n+2m-i-j \choose m-j} {i+j \choose i} \\
\qquad= {n+m \choose m}^2 {n+2m \choose n} {2m \choose m}^3.
\end{gather*}

\section{Non-intersecting Brownian bridges}\label{nibb}

Let $M$ (resp.\ $H$) be the maximal height of $N$ non-intersecting reflected Brownian bridges starting and ending
at zero (resp.\ non-intersecting Brownian excursions).
It is well known~\cite{f,fms,kat,lie} that the distribution functions of $H$ and $M$ are directly related to the
partition functions of certain discrete Gaussian ensembles.
To be more precise, let
\[\Delta^C_N(x)=\prod_{1\le i\le N} x_i \prod_{1\le i<j\le N} \big(x_i^2-x_j^2\big), \qquad \Delta^D_N(x)= \prod_{1\le i<j\le N} \big(x_i^2-x_j^2\big).\]
Then
\[P(H\le h)=A_{N} h^{-2N^2-N} \sum_{x\in \Z^N} \Delta^C_N(x)^2 {\rm e}^{-\pi^2 \sum_i x_i^2/(2h^2)},\]
and
\[P(M\le h)=A'_{N} h^{-2N^2+N} \sum_{x\in (\Z-1/2)^N} \Delta^D_N(x)^2 {\rm e}^{-\pi^2 \sum_i x_i^2/(2h^2)},\]
where $A_{N}$ and $A'_{N}$ are normalisation constants.
In~\cite[Theorem 4]{kat}, it is shown, using Jacobi's theta function identity, that
the first expression may be written as
\begin{equation}\label{Hf}
P(H\le h) = B_N \sum_{\lambda\in \Z^N} \det\big[ H_{2(i+j-1)}\big(\lambda_i \sqrt{2} h\big) \big]_{i,j=1}^N {\rm e}^{- 2h^2 \sum_i \lambda_i^2},\end{equation}
where $H_n(x)$ are the Hermite polynomials and $B_N$ is a~normalisation constant.
The proof given in~\cite{kat} is easily modified to give an analogous formula for $M$
\begin{equation}\label{Mf}
P(M\le h) = B'_N \sum_{\lambda\in \Z^N} \det\big[ H_{2(i+j-2)}\big(\lambda_i\sqrt{2} h\big) \big]_{i,j=1}^N {\rm e}^{-2 h^2 \sum_i \lambda_i^2 },\end{equation}
where $B'_N$ is a~normalisation constant.
Let us define, for $\lambda\in\Z^N$ and $t\ge 0$,
\[P_{\lambda}(t) = \frac1{N!} \sum_{\sigma\in S_N} \tilde P_{\sigma \lambda}(t),\qquad Q_{\lambda}(t) = \frac1{N!} \sum_{\sigma\in S_N} \tilde Q_{\sigma \lambda}(t),\]
where
\[\tilde Q_\lambda(t) = B_N \det\big[ H_{2(i+j-1)}\big(\lambda_i \sqrt{t} \big) \big]_{i,j=1}^N,\]
and \[\tilde P_\lambda(t) = B'_N \det\big[ H_{2(i+j-2)}\big(\lambda_i \sqrt{t} \big) \big]_{i,j=1}^N. \]
Then~\eqref{Hf} and~\eqref{Mf} may be written as
\[P\big(2M^2\le t\big) = \sum_{\lambda\in \Z^N} P_{\lambda}(t) {\rm e}^{- \sum_i \lambda_i^2 t},\qquad P\big(2H^2\le t\big) = \sum_{\lambda\in \Z^N} Q_{\lambda}(t) {\rm e}^{- \sum_i \lambda_i^2 t}.\]
For $N=1$, writing $\lambda=k$, these symmetrised coefficients are given by
\[P_k(t)=(-1)^k,\qquad Q_k(t)=1-2k^2t,\]
and for $N=2$, writing $\lambda=(k,l)$, by
\begin{gather*}
P_{k,l}(t) = (-1)^{k+l}\big[ 1 - 2\big(k^2+l^2\big) t + \big(k^2-l^2\big)^2 t^2\big],\\
Q_{k,l}(t) = 1 - 4\big(k^2+l^2\big) t + 3\big(k^2+l^2\big)^2t^2 -\frac23\big(k^2+l^2\big)^3t^3+\frac43 k^2l^2\big(k^2-l^2\big)^2 t^4.
\end{gather*}
In general, for $\lambda\in\Z^N$, $P_\lambda(t)$ has degree at most $N(N-1)$ and $Q_\lambda(t)$
has degree at most $N^2$.

\section{Transition probabilities and absorption times}\label{tpat}

Let $h^r=h^{r,2}$ and recall that
\[ h^{r} = \sum_{i=1}^r L_{n_i} - P(n),\]
where
\begin{equation}\label{P}
P(n)=\sum_{i=1}^r n_i^2 - \sum_{i=1}^{r-1} n_i n_{i+1}
\end{equation}
and $L_k$ denotes the partial backward shift operator defined by $L_k f(k) = f(k-1)$,
so, for example, if $f(n)$ is a~function of $n\in\Z_+^r$ (or $\C^r$) then
\[L_{n_i}f(n)=f(n-e_i)=f(n_1,\dots,n_{i-1},n_i-1,n_{i+1},\dots,n_r).\]
We also recall the notation $D_k = L_k-I$ for the partial backward difference operator
$D_k f(k) = f(k-1) - f(k)$, so, for example, if $f(n)$ is a~function of $n\in\Z_+^r$ (or $\C^r$) then
\[D_{n_i}f(n)=f(n-e_i)-f(n).\]

Since $a_r$ is strictly positive on $\Z_+^r$ and satisfies $h^r a_r=0$, the corresponding Doob transform
\[\cL=a_r(n)^{-1} \circ h^r \circ a_r(n) = \sum_{i=1}^r \frac{a_r(n-e_i)}{a_r(n)} D_{n_i}\]
generates a~continuous time Markov chain $X_t,\ t\ge 0$ on $\Z_+^r$.
In~\cite{noc23}, it was shown that this Markov chain has a~unique entrance
law starting from $+\infty$.

Denote by $\P_n$ the law of the chain starting from
$n\in\Z_+^r$ and by $\P$ the law of chain started from~$+\infty$.
Write $\E_n$ and $\E$ for the corresponding expectations.
Denote the transition probabilities by~$p_t(n,n')=\P_n(X_t=n')$,
and set
\[F_n(t)=p_t(n,0)=\P_n(T\le t),\qquad F(t)=\lim_{n\to\infty} p_t(n,0)=\P(T\le t),\]
where
$T=\inf\{t\ge 0\mid X_t=0\}$ is the absorption
time for the chain.

For any starting value $n\in\Z_+^r$, the chain remains in the finite set
$E_n=\{m\in\Z_+^r\mid m\le n\}$.
Thus, for each $n'\in\Z_+^r$, $p_t(n,n')$ is the unique solution to the backward equation
\[ \frac{\rm d}{{\rm d}t} p_t(n,n') = \cL_n p_t(n,n'),\qquad p_{0}(n,n')=\delta_{nn'}. \]
In particular, $F_n(t)$ is the unique solution to
$F_n'(t) = \cL_n F_n(t)$, $ F_n(0)=\delta_{n0}$.
Similarly, for each~${n\in\Z_+^r}$, $p_t(n,n')$ is the unique solution to the forward equation
\[ \frac{\rm d}{{\rm d}t} p_t(n,n') = \cL^*_{n'} p_t(n,n'),\qquad p_{0}(n,n')=\delta_{nn'}, \]
where $\cL^*$ denotes the formal adjoint of $\cL$.

For $s\ge 0$, the generating function
\[ G_n(s)=\E_n {\rm e}^{-s T} = \int_0^\infty s {\rm e}^{-s t} F_n(t) {\rm d}t\]
satisfies the recursion
$\cL_n G_n(s)=s G_n(s)$, that is,
\begin{equation}\label{g-rec}
G_n(s)=\frac1{s+P(n)} \sum_{i=1}^r \frac{a_r(n-e_i)}{a_r(n)} G_{n-e_i}(s),
\end{equation}
with $G_0(s)=1$ and the convention $G_n(s)=0$ for $n\notin\Z_+^r$.
Denote $G(s)=\E {\rm e}^{-s T}$.

\subsection[The case r=1]{The case $\boldsymbol{r=1}$}\label{r1}

When $r=1$, we have $\cL=n^2 D_n$. The absorption time starting from $n$ has the same law as the random variable
\begin{equation}\label{sn} S_n=\sum_{k=1}^n \frac{e_k}{k^2},\end{equation}
where $e_1,e_2,\dots$ is a~sequence of independent standard exponential random variables.
This gives the formulas
\smash{$G_n(s) = \prod_{k=1}^n \frac{k^2}{s+k^2}$}, \smash{$ G(s) = \frac{\pi\sqrt{s}}{\sinh \pi\sqrt{s}}$},
and \begin{equation}\label{at1} F(t)=\sum_{k\in\Z} (-1)^k {\rm e}^{-k^2 t}.\end{equation}

Let
\[\varphi_k(n)= \frac{n!^2}{(n-k)!(n+k)!},\qquad \varphi^*_k(m) = (-1)^k \frac{(-k)_m(k)_m}{m!^2},\]
where $(x)_n=\Gamma(x+n)/\Gamma(x)$.
One can easily check that $\cL \varphi_k = -k^2 \varphi_k$ and
\begin{equation}\label{or} \sum_{k\in\Z} \varphi_k(n) \varphi^*_k(m) =\delta_{nm}.\end{equation}
It also holds that $\cL^* \varphi^*_k = -k^2 \tilde\varphi_k$ and, moreover,
\begin{equation}\label{cp} \sum_{n=0}^\infty \varphi_k(n) \varphi^*_l(n) = \frac{\delta_{kl}+\delta_{k,-l}}{2}.\end{equation}
The transition probabilities therefore given by
\begin{equation}\label{kmg1}
p_t(n,m)=\sum_{k\in\Z} \varphi_k(n) \varphi^*_k(m) {\rm e}^{-k^2 t}.
\end{equation}
Note that, since $\varphi_k(n)=0$ for $|k|>n$ and $\varphi^*_k(m) =0$ for $|k|<m$,
the sums in~\eqref{or}, \eqref{cp} and~\eqref{kmg1} are all terminating.
Setting $m=0$ in~\eqref{kmg1} yields
\[F_n(t) = \sum_{k=-n}^n (-1)^k \varphi_k(n) {\rm e}^{-k^2 t}\]
and, letting $n\to\infty$, we recover~\eqref{at1}.

\subsection{The general case}

To formulate a~spectral expansion for the transition probabilities in the general case,
we need to make some assumptions. For $\nu\in\C^{r+1}_0$, recall the notation $\nu'\in\C^r$
defined by $\nu'_i=\nu_1+\cdots+\nu_i$, $i=1,\dots,r$. For $\nu\in\C^{r+1}_0$ and $n\in\nu'+\Z_+^r$,
let $ f_{r,\nu}(n)$ be defined by~\eqref{def-fr}.
Recall that
\begin{gather}
 f_{1,\nu}(n) =\frac1{\Gamma(n-\nu_1+1)\Gamma(n-\nu_2+1)},\nonumber\\
f_{2,\nu}(n,m)=\frac{\Gamma(n+m+1)}{\prod_{i=1}^3 \Gamma(n-\nu_i+1)\Gamma(m+\nu_i+1)}.\label{f12}
\end{gather}
By Theorem~\ref{jef-f}, setting $\lambda_\nu=e_2(\nu)$, these satisfy
\begin{equation}\label{ee-fr} h^r f_{r,\nu} = \lambda_\nu f_{r,\nu}.\end{equation}
The formulas~\eqref{f12}
may be used to define meromorphic extensions of $f_{1,\nu}(n)$ and $f_{2,\nu}(n,m)$
to complex values of $n$ and $m$, and it is easily checked that these extensions satisfy the eigenvalue equation~\eqref{ee-fr}. By Theorem~\ref{ff}, if
\[ n_i=\begin{cases}
 in_2/2, & i \mbox{ \rm even},\\ n_1+(i-1)n_2/2, & i \mbox{ \rm odd},\end{cases}\]
for $i\ge 1$, that is,
\begin{equation}\label{spec}
 n_1,n_2,n_3,\ldots =n_1,n_2,n_1+n_2,2n_2,n_1+2n_2,3n_2,\ldots \end{equation}
then for $\nu\in\C_0^{r+1}$, if $(n_1,\dots,n_r)\in \nu'+\Z_+^{r}$ and $\operatorname{Re}(n_{r+1})>-1$, we have
\begin{equation}\label{fff} f_{r,\nu}(n_1,\dots,n_r) = \Gamma(n_{r+1}+1) \prod_i \frac1{\Gamma(n_1-\nu_i+1) }\prod_{i<j} \frac1{\Gamma(n_2-\nu_i-\nu_j+1)}.\end{equation}
In general, we will assume that $f_{r,\nu}(n)$ extends to a~meromorphic function of $n\in\C^r$
which satisfies both the eigenvalue equation~\eqref{ee-fr} and the factorisation property~\eqref{fff}.
Given this, we also extend the definition of $a_r(n)=f_{r,0}(n)$.

For $\nu\in\C^{r+1}_0$ and $n\in\C^r$, let
$\beta_{r,\nu}(n) = a_r(n)^{-1} f_{r,\nu}(n)$, $ \beta^*_{r,\nu}(n) = h(\nu) a_r(n) M_{r,-\nu}(n)$,
where \[h(\nu)=\prod_{1\le i<j\le r+1}(\nu_i-\nu_j).\]
These satisfy $\cL\beta_{r,\nu}=\lambda_\nu\beta_{r,\nu}$ and, by~\eqref{jepm},
$\cL^*\beta^*_{r,\nu}=\lambda_\nu\beta^*_{r,\nu}$.

Note that for $\nu\in\C^{r+1}_0$,
\[\lambda_\nu\equiv e_2(\nu)=-\frac12 \sum_{i=1}^{r+1}\nu_i^2=-P(\nu'),\]
where $P$ is defined by~\eqref{P}. For $\nu\in\C_{0}^{r+1}$, $x\in\C^r$ and $n\in\Z_+^r$, define
\[C_{\nu'}(x,n,t)=\beta_{r,\nu}(x+n) \beta^*_{r,\nu}(x) {\rm e}^{-P(\nu') t}.\]

Our proposed formula for the transition probabilities is, for $n,n'\in\Z_+^r$,
\begin{equation}\label{pf} p_t(n'+n,n')=\sum_{p=0}^\infty c_p(n',n,t) {\rm e}^{-pt},\end{equation}
where
\[ c_p(n',n,t) = \lim_{x\to n'}\lim_{u\to x} \prod_{i=1}^r(x_i-u_i)
\sum_{\kappa'\in\Z_+^r\colon P(\kappa')=p} C_{u+\kappa'}(x,n,t) {\rm e}^{p t}.\]

When $r=1$, this reduces to~\eqref{kmg1}. When $r=2$, it is also correct, as will be outlined in the next section.
In general, by construction, assuming it is well defined, this expression should satisfy the backward
and forward equations. The boundary condition $p_{0}(n'+n,n')=\delta_{n0}$ might be difficult to
establish in general, but can in principle be verified on a~case by case basis for $n$ and $n'$ of the form~\eqref{spec},
using the factorisations~\eqref{fff} and~\eqref{mf}.

For $n'$ and $n$ of the form~\eqref{spec}, we can also take $x$ to be of the form~\eqref{spec}.
In this case, using the factorisations~\eqref{fff} and~\eqref{mf}, the formula~\eqref{pf} becomes
more explicit and we state it here as a~conjecture.

\begin{conj}
\begin{equation}\label{atr-nn}
p_t(n'+n,n')=\sum_{p=0}^\infty \tilde c_p(n_1',n_2',n_1,n_2,t) {\rm e}^{-pt},\end{equation}
where
\begin{gather*}
 \tilde c_p(n_1',n_2',n_1,n_2,t)=
 \lim_{(x_1,x_2)\to (n_1',n_2')}\lim_{u\to x} \prod_{i=1}^r(x_i-u_i)
\sum_{\kappa'\in\Z_+^r\colon P(\kappa')=p} \tilde C_{u+\kappa'}(x_1,x_2,n_1,n_2,t) {\rm e}^{pt},\\
\tilde C_{\nu'}(x_1,x_2,n_1,n_2,t) = h(\nu) (1+x_1)_{n_1}^{r+1} (1+x_2)_{n_2}^{r(r+1)/2} x_{r+1} \\
\phantom{\tilde C_{\nu'}(x_1,x_2,n_1,n_2,t) =}{} \times \prod_i \frac1{(x_1-\nu_i)_{n_1+1}} \prod_{ i<j}\frac1{(x_2-\nu_i-\nu_j)_{n_2+1}} {\rm e}^{-P(\nu') t}.
\end{gather*}
\end{conj}

In the case $n'=0$, taking $x_1=x_2=\xi$, we can write this as follows.

\begin{conj} For $n\in\Z_+^r$ of the form~\eqref{spec},
\begin{equation}\label{atr-n}
F_n(t)=\sum_{p=0}^\infty \tilde c_p(n_1,n_2,t) {\rm e}^{-pt},\end{equation}
where, writing $x=(x_1,x_2,\dots,x_r)=(\xi,\xi,2\xi,2\xi,\dots)$,
\begin{gather*}
 \tilde c_p(n_1,n_2,t)=
 \lim_{\xi\to 0}\lim_{u\to x} \prod_{i=1}^r(x_i-u_i)
\sum_{\kappa'\in\Z_+^r\colon P(\kappa')=p} \tilde C_{u+\kappa'}(\xi,n_1,n_2,t) {\rm e}^{pt},\\
\tilde C_{\nu'}(\xi,n_1,n_2,t)= h(\nu) n_1!^{r+1} n_2!^{r(r+1)/2} x_{r+1}
\prod_i \frac1{(\xi-\nu_i)_{n_1+1}} \prod_{ i<j}\frac1{(\xi-\nu_i-\nu_j)_{n_2+1}} {\rm e}^{-P(\nu') t}.
\end{gather*}
\end{conj}

Now, since $x+n$ also has the form~\eqref{spec}, it follows from~\eqref{fff} that $\beta_{r,\nu}(x+n)\to 1$
when $n_1,n_2\to \infty$. Given that there is a~unique entrance law from $+\infty$, we can therefore let
$n_1,n_2\to \infty$ to obtain the following.
\begin{conj}
\begin{equation}\label{atr}
F(t)=\sum_{p=0}^\infty c_p(t) {\rm e}^{-pt},\end{equation}
where, writing $x=(x_1,x_2,\dots,x_r)=(\xi,\xi,2\xi,2\xi,\dots)$,
\begin{gather*}
c_p(t)= \lim_{\xi\to 0}\lim_{u\to x} \prod_{i=1}^r(x_i-u_i)
\sum_{\kappa'\in\Z_+^r\colon P(\kappa')=p} C_{u+\kappa'}(\xi,t) {\rm e}^{pt},\\
C_{\nu'}(\xi,t)= h(\nu) \Gamma(x_{r+1})^{-1} \prod_i\Gamma(\xi-\nu_i) \prod_{ i<j}\Gamma(\xi-\nu_i-\nu_j) {\rm e}^{-P(\nu') t}.
\end{gather*}
\end{conj}

For computations, it will be convenient to `split' the coefficients $c_p(t)$ as follows.
For $\kappa'\in\Z_+^r$, denote
$\kappa=\big(\kappa_1',\kappa_2'-\kappa_1',\dots,\kappa_r'-\kappa_{r-1}',-\kappa_r'\big)$, $ \kappa^\dagger=(\kappa_{r+1},\kappa_r,\dots,\kappa_1)$.
Define an equivalence relation on $\Z_+^r$ by $\rho'\sim\kappa'$ if $\rho=\sigma\kappa$ or $\rho=\sigma\kappa^\dagger$
for some $\sigma\in S_{r+1}$ and denote the corresponding equivalence classes by
$S_{\kappa'}=\left\{\rho'\in\Z_+^r\mid \rho'\sim\kappa'\right\}$.
Note that if $\rho'\sim\kappa'$ then $P(\rho')=P\big(\kappa'\big)$.
The formula~\eqref{atr} may be written as
\begin{equation}\label{atr-ck} F(t)=\sum_{\kappa'\in \Z_+^r/{\sim}} c_{\kappa'}(t) {\rm e}^{-P(\kappa')t},\end{equation}
where, writing $x=(x_1,x_2,\dots,x_r)=(\xi,\xi,2\xi,2\xi,\dots)$,
\[c_{\kappa'}(t)= \lim_{\xi\to 0}\lim_{u\to x} \prod_{i=1}^r (x_i-u_i)
\sum_{\rho' \in S_{\kappa'}} C_{u+\rho'}(\xi,t) {\rm e}^{P(\rho')t}.\]
Note that the coefficients $c_p(t)$ are then given as
\[c_p(t)=\sum_{\kappa'\in \Z_+^r/{\sim}\colon P(\kappa')=p} c_{\kappa'}(t).\]

\subsection[The case r=2]{The case $\boldsymbol{r=2}$}

Writing $(n_1,n_2)=(n,m)$, the generator is given by
\[\cL = \frac{n^3}{n+m} D_n + \frac{m^3}{n+m} D_m.\]
In~\cite[Proposition 4.4]{noc23}, it was shown that the absorption time starting from $(n,m)$ has the same law as
$S_n+S'_m$, where $S_n$ is defined by~\eqref{sn} and $S_m'$ is an independent copy of $S_m$.
This gives the formulas
\[G_{n,m}(s) = \prod_{k=1}^n \frac{k^2}{s+k^2} \prod_{l=1}^m \frac{l^2}{s+l^2},\qquad
G(s) = \left(\frac{\pi\sqrt{s}}{\sinh \pi\sqrt{s}}\right)^2,\]
and
\begin{equation}\label{at2}
F(t)= \sum_{k\in\Z} \big(1-2k^2t\big) {\rm e}^{-k^2 t}.
\end{equation}

Writing $(x_1,x_2)=(x,y)$, the formula~\eqref{atr-nn} is correct and may be written as follows.

\begin{thm}
\begin{align}
p_t((n'+n,m'+m),(n',m')) ={} &
\frac{(n'+n)!^3 (m'+m)!^3 }{n'!^3 m'!^3} \nonumber \\
&\times \lim_{(x,y)\to (n',m')} \sum_{k=0}^n \sum_{l=0}^m
T^{n,m}_{k,l}(x,y) {\rm e}^{-P(x+k,y+l)t},\label{pf2}
\end{align}
where $P$ is defined by~\eqref{P} and
\begin{align*}
T^{n,m}_{k,l}(x,y)={}
& \frac{x+y}{x+y+k+l} \prod_{\substack{a=0\\ a\ne k}}^{n} \frac{1}{(a-k)(2x-y+a+k-l)(x+y+a+l)} \\
&\times \prod_{\substack{b=0\\ b\ne l}}^{m} \frac1{(b-l)(2y-x+b+l-k)(x+y+b+k)},
\end{align*}
with the convention that empty products are equal to one.
\end{thm}
\begin{proof}
The formula~\eqref{pf2} is clearly equal to one when $t=0$ and $(n,m)= (0,0)$.
One can also show that it vanishes at $t=0$ when $(n,m)\ne (0,0)$.
In fact, if $(n,m)\ne (0,0)$,
\[\sum_{k=0}^n \sum_{l=0}^m T^{n,m}_{k,l}(x,y) =0.\]
This follows from the more general identity
\begin{gather}
\sum_{k=0}^n \sum_{l=0}^m
\frac1{\lambda_k+\mu_l}
 \prod_{\substack{a=0\\ a\ne k}}^{n} \frac{1}{(\lambda_a-\lambda_k)(\lambda_a+\lambda_k-\mu_l)(\lambda_a+\mu_l)} \nonumber \\
\qquad \times \prod_{\substack{b=0\\ b\ne l}}^{m} \frac1{(\mu_b-\mu_l)(\mu_b+\mu_l-\lambda_k)(\mu_b+\lambda_k)} =0,\label{identity}
\end{gather}
with the convention that empty products are equal to one and the assumption that
each term in the sum is finite. A~proof of~\eqref{identity} is given in Appendix~\ref{AA}.
\end{proof}

For small values of $(n',m')$ and $(n,m)$ this is easily computed using \textsc{Mathematica}.
For example, and writing $q={\rm e}^{-t}$, it yields
\begin{gather*}
p_t((2,2),(1,1))=\frac1{9}\big( 16 q- 48q^3+32 q^4\big),\\
p_t((2,2),(1,0))=\frac8{27}\bigl( -2q+2q^4-3( q+ q^4) \ln q\bigr),\\
p_t((2,2),(0,0))=\frac1{27}\big( 27-16 q - 11q^4+12 (4 q+q^4) \ln q\big).\end{gather*}

In the case $n'=m'=0$, the formula~\eqref{pf2} is easily computed for all values of $(n,m)$
using \textsc{Mathematica}.
Only the terms corresponding to $(k,l)$ of the form $(k,k)$, $(k,0)$ and $(0,k)$ contribute,
and the formula yields
\begin{equation}\label{ans2} F_{n,m} ( t) = \sum_{k\in\Z} \big[ \tilde H_k(n,m)- 2k^2 H_k(n,m) t \big] {\rm e}^{-k^2 t},\end{equation}
where
\[H_k(n,m)=\frac{n!^2m!^2}{(n-k)!(n+k)!(m-k)!(m+k)!},\]
and $\tilde H_k(n,m)$ is defined as follows. If $|k|\le \min\{n, m\}$, then
\[\tilde H_k(n,m)=\left[ 1+k(H_{n-k}-H_{n+k}+H_{m-k}-H_{m+k}) \right] H_k(n,m),\]
where $H_n$ are the harmonic numbers
$H_n=\sum_{j=1}^n \frac1{j}$,
with the convention $H_0=0$. If $n\ge k > m$, then
\[\tilde H_k(n,m)=(-1)^{m+k} k \Gamma(k-m) \frac{n!^2m!^2}{(n-k)!(n+k)!(m+k)!}.\]
If $|k|>n\vee m$, then $\tilde H_k(n,m)=0$. The remaining cases are defined by
\[\tilde H_{-k}(n,m)=\tilde H_k(n,m)=\tilde H_k(m,n).\]
We remark that the formula~\eqref{ans2} may also be verified directly by showing that
\[\cL H_k = -k^2 H_k, \qquad \cL \tilde H_k = -k^2 \tilde H_k - 2k^2 H_k,\]
and
\[\sum_{k\in\Z} \tilde H_k(n,m) = \delta_{n0}\delta_{m0}.\]

Now let us consider the formula~\eqref{atr}. In this case it may be written as
\begin{equation}\label{F2} F(t)=\sum_{p=0}^\infty c_p(t) {\rm e}^{-pt},\end{equation}
where
\begin{gather*}
c_p(t)= \lim_{x\to 0} \sum_{(k,l)\in\Z_+^2\colon P(k,l)=p} T_{k,l}(x) {\rm e}^{pt},
\qquad
T_{k,l}(x)= \lim_{u\to x} (x-u)^2 C_{u+k,u+l}(x,t),\\
C_{\nu'}(x,t)= h(\nu) \Gamma(2x)^{-1} \prod_i\Gamma(x-\nu_i) \Gamma(x+\nu_i) {\rm e}^{-P(\nu') t}.\end{gather*}

Let $q={\rm e}^{-t}$. We first note that
\[T_{0,0}(x)=2x^3 \Gamma (x)^2 \Gamma (2 x) q^{x^2},\qquad
c_0(t)=\lim_{x\to 0} T_{0,0}(x)=1.\]
Next consider $(k,l)$ of the form $(k,k)$, $(k,0)$ and $(0,k)$, where $k>0$.
We compute
\begin{gather*}
T_{k,k}(x)=\frac{2 (k+x)^3 \Gamma (x)^2 \Gamma (k+2 x)^2}{(k!)^2\Gamma(2x)} q^{(k+x)^2},\\
T_{k,0}(x)=T_{0,k}(x)= \frac{(-1)^k (2 k+x) \Gamma (2 x) \Gamma (-k+x+1) \Gamma (k+x) \Gamma
 (k+2 x+1)}{k!\Gamma(2x)} q^{k^2+k x+x^2}.\end{gather*}
Combining these gives
\begin{align*}
S_k(x) :={}& T_{k,k}(x)+T_{k,0}(x)+T_{0,k}(x) \\
={}& \frac{2 \Gamma (k+2 x) }{k!^2\Gamma(2x)} q^{k^2+k x+x^2}
\big[ (-1)^k k! (2 k+x) (k+2 x) \Gamma (2 x)
 \Gamma (-k+x+1) \Gamma (k+x) \\
 & +(k+x)^3 \Gamma (x)^2 \Gamma (k+2x) q^{kx } \big].
\end{align*}
Letting $x\to 0$, this gives
\[c_{k^2}(t)=\lim_{x\to 0} S_k(x) = 2\big(1+2k^2\ln q\big)q^{k^2}=2\big(1-2k^2t\big){\rm e}^{-k^2 t}.\]

Now suppose $k>l>0$. Then
\begin{align*}
T_{k,l}(x)={}& \frac{ (-1)^{k+l}}{k! l! \Gamma(2x)} (2 k-l+x) (-k+2 l+x) (k+l+2 x) \\
&\times \Gamma (k+2 x) \Gamma (l+2 x) \Gamma (k-l+x) \Gamma (-k+l+x) q^{-(k+x)
 (l+x)+(k+x)^2+(l+x)^2},
\end{align*}
and
\[T_{k,l}:=\lim_{x\to 0} T_{k,l}(x)=-\frac{4 (k-2 l) (2 k-l) (k+l) }{k l (k-l)} q^{k^2-k l+l^2}.\]
Now, if $k>l>0$ then $k>k-l>0$ and, moreover,
$T_{k,l}=-T_{k,k-l}$.
It follows that the combined contribution of such terms is zero.
Essentially the same calculation shows that the
combined contribution of terms corresponding to $(k,l)$ with $l>k>0$ is also zero.
The formula~\eqref{F2} therefore yields
\[F(t)=1+2\sum_{k=1}^\infty \big(1-2k^2t\big){\rm e}^{-k^2 t} = \sum_{k\in\Z} \big(1-2k^2t\big){\rm e}^{-k^2 t},\]
in agreement with~\eqref{at2}. We note that the above computation is also easily verified
using \textsc{Mathematica}.

\subsection[The case r=3]{The case $\boldsymbol{r=3}$}

The generating function $G_n(s)$ may be computed for small values of $n$ using the recursion~\eqref{g-rec}.
For example, using \textsc{Maple} we obtain (note that $G_{nmp}=G_{pmn}$)
\begin{gather*}
G_{100}(s)=G_{010}(s)=\frac1{s+1},\qquad
G_{011}(s)= \frac1{(s+1)^2},\qquad G_{101}(s)=\frac2{(s+1)(s+2)},\\
G_{111}(s)=\frac{2(3s+5)}{5(s+1)^3(s+2)},\qquad G_{120}(s)=G_{210}(s)=\frac{4}{(s+1)^2(s+4)},\\
G_{130}(s)=\frac{36}{(s+1)^2(s+4)(s+9)},\qquad
G_{121}(s)=\frac{16(3s^2+13s+13)}{13(s+1)^3(s+2)^2(s+4)},\\
G_{112}(s)=\frac{2(3s+5)}{(s+1)^3(s+2)(s+5)},\qquad
G_{123}(s)=\frac {32(5 {s}^{3}+45 {s}^{2}+122 s+100)}{ (s+1)^3(s+2)^2(s+4)^2(s+5)(s+10)},\\
G_{213}(s)=
\frac{32(15s^3 + 180s^2 + 647s + 650)}{(s + 1)^3(s + 2)(s + 4)(s + 5)^2(s + 8)(s + 13)},\\
G_{224}(s)=\frac{512(35s^6 + 1120s^5 + 13895s^4 + 85550s^3 + 274788s^2 + 433320s + 260000)}
{(s + 20)(s + 8)(s + 10)(s + 5)^2(s + 2)^2(s + 13)(s + 4)^3(s + 1)^3}.\end{gather*}

Since $G_n(s)/s$ is the Laplace transform of $F_n(t)$, these can be inverted (again using \textsc{Maple})
to obtain, writing $q={\rm e}^{-t}$,
\begin{gather}
F_{100}(t)=F_{010}(t)=1-q,\qquad
F_{011}(t)= 1+(\ln q-1)q,\qquad F_{101}(t)=1-2q+q^2,\label{f2}\\
F_{111}(t)=\frac15\big(5+\big(-4+6\ln q -2 ( \ln q )^2\big)q-q^2 \big),\nonumber \\
F_{120}(t)=F_{210}(t)=\frac19\big( 9 +(-8+12\ln q ) q -{q}^{4} \big),\nonumber \\
F_{130}(t)=\frac1{80}\big( 80+(-65+120\ln q) q -16 {q}^{4} +{q}^{9} \big),\nonumber \\
F_{121}(t)=\frac1{39}\big( 39 + \big(-24(\ln q)^2 + 48\ln q - 16\big) q + (-24 + 12\ln q)q^2 + q^4 \big),\nonumber \\
F_{112}(t)=\frac1{48}\big( 48 + \big(-24(\ln q)^2 + 60\ln q - 33\big) q - 16 q^2 + q^5 \big),\label{f112}
\\
F_{123}(t)=\frac1{1620} \big( 1620 + \big(-1440(\ln q)^2 + 1680\ln q - 220\big) q + (-1665 + 1080\ln q)q^2 \nonumber \\
\phantom{F_{123}(t)=}{} + (300 - 240\ln q)q^4 - 36q^5 + q^{10} \big),\label{f123}\\
F_{213}(t)= \frac1{2160}
\big( 2160 + \big(-1440(\ln q)^2 + 2640\ln q - 1160\big) q - 1280q^2\nonumber\\
\phantom{F_{213}(t)=}{} + 160 q^4 + (135 - 120\ln q) q^5 - 16 q^8 + q^{13} \big),\label{f213}
\\
F_{224}(t)= \frac1{680400}
\big( 680400 + \big(-806400(\ln q)^2 + 403200 \ln q - 22400\big) q \nonumber\\
\phantom{F_{224}(t)=}{} + (-974400 + 806400\ln q)q^2 + \big(201600(\ln q)^2 - 529200\ln q + 333725\big) q^4 \nonumber\\
\phantom{F_{224}(t)=}{} + (40320\ln q - 18480)q^5 + 1680q^8 - 448q^{10} - 80q^{13}+3q^{20} \big).\label{f224}
\end{gather}

For the cases where $p=n+m$, the formula~\eqref{atr-n} may be computed using \textsc{Mathematica}
and is found to agree with~\eqref{f2}, \eqref{f112}, \eqref{f123}, \eqref{f213} and~\eqref{f224},
and in fact we have similarly verified that it is correct for all $n,m\le 4$. We have also verified,
using \textsc{Mathematica}, that the formula~\eqref{atr-n} satisfies the boundary condition
$F_{n,m,n+m}(0)=0$ for all $n,m\le 10$ with $(n,m)\ne (0,0)$.

We note that the above formulas are consistent with the ansatz (cf.\ \eqref{ans2})
\begin{equation}\label{ans3}
F_n(t)=\sum_{k,l\in\Z} P^n_{k,l}(t) {\rm e}^{-(k^2+l^2)t},
\end{equation}
where $P^n_{k,l}(t)$ are polynomials of degree at most 2 (depending on $k$, $l$ only through $|k|$, $|l|$)
which satisfy
\[\frac{\rm d}{{\rm d}t} P^n_{k,l}(t) = \big(\cL_n+k^2+l^2\big) P^n_{k,l}(t),\qquad \sum_{k,l\in\Z} P^n_{k,l}(0)=\delta_{n0}. \]

Now let us consider the formula~\eqref{atr}. For any given $p$, the coefficient $c_p(t)$ is easily
computed using \textsc{Mathematica}. The first few coefficients are given by
\begin{gather*}
c_0(t)=1,\qquad c_1(t)=-4(1-t)^2,\qquad c_2(t)=4 (1-4t),\\
c_3(t)=0, \qquad c_4(t)=4(1-4t)^2,\qquad c_5(t)=-8(1-t)(1-9t),\qquad c_6(t)=0.\end{gather*}
The formula~\eqref{atr-ck} may be written as
\[F(t)=\sum_{\kappa'\in \Z_+^3/\sim} c_{\kappa'}(t) {\rm e}^{-P(\kappa')t},\]
where, writing $x=(x_1,x_2,x_3)=(x,x,2x)$,
\begin{gather}
c_{\kappa'}(t)= \lim_{x\to 0}\lim_{u\to x} \prod_{i=1}^3 (x_i-u_i)
\sum_{\rho' \in S_{\kappa'}} C_{u+\rho'}(x,t) {\rm e}^{P(\rho')t},\nonumber\\
C_{\nu'}(x,t)= h(\nu) \Gamma(2x)^{-1} \prod_{i=1}^4 \Gamma(x-\nu_i) \prod_{1\le i<j\le 4}\Gamma(x-\nu_i-\nu_j) {\rm e}^{-P(\nu') t}.\label{ck3}
\end{gather}

In the following, let $\kappa'=(k,m,l)\in\Z_+^3$ and consider different cases. In each case, the coefficient $c_{\kappa'}(t)$
defined by~\eqref{ck3} is computed using \textsc{Mathematica}.

{\em Case $(i)$:} $k>l>0$ and $m>k+l$.
Then $S_{\kappa'}$ consists of the six distinct elements
\[\begin{array}{|c|c|}
\hline
\rho' & \rho \\ \hline
(k,m,l) & (k,m-k,l-m,-l) \\
(k,m,m-l) & (k,m-k,-l,l-m) \\
(k,k-l,m-l) & (k,-l,m-k,l-m) \\ \hdashline
(l,m,k) & (l,m-l,k-m,-k) \\
(m-l,m,k) & (m-l,l,k-m,-k) \\
(m-l,k-l,k) & (m-l,k-m,l,-k) \\
\hline
\end{array}
\]
and additionally, if $m\ne 2k$,
\[\begin{array}{|c|c|} \hline
\rho' & \rho \\ \hline
(m-k,m,l) & (m-k,k,l-m,-l) \\
(m-k,m,m-l) & (m-k,k,-l,l-m) \\
(m-k,m-k-l,m-l) & (m-k,-l,k,l-m) \\ \hdashline
(l,m,m-k) & (l,m-l,-k,k-m) \\
(m-l,m,m-k) & (m-l,l,-k,k-m) \\
(m-l,m-k-l,m-k) & (m-l,-k,l,k-m) \\
\hline
\end{array}
\]
Either way we find that $c_{\kappa'}(t)=0$.

{\em Case $(ii)$:} $k>l=0$ and $m>k$.
Then $S_{\kappa'}$ consists of
\[\begin{array}{|c|c|}
\hline
\rho' & \rho \\ \hline
(k,m,0) & (k,m-k,-m,0) \\[-1.0pt]
(k,m,m) & (k,m-k,0,-m) \\[-1.0pt]
(k,k,m) & (k,0,m-k,-m) \\[-1.0pt]
(0,k,m) & (0,k,m-k,-m) \\[-1.0pt]
\hdashline
(0,m,k) & (0,m,k-m,-k) \\[-1.0pt]
(m,m,k) & (m,0,k-m,-k) \\[-1.0pt]
(m,k,k) & (m,k-m,0,-k) \\[-1.0pt]
(m,k,0) & (m,k-m,-k,0)\\[-1.0pt]
\hline
\end{array}
\]
and additionally, if $m\ne 2k$,
\[\begin{array}{|c|c|} \hline
\rho' & \rho \\ \hline
(m-k,m,0) & (m-k,k,-m,0) \\[-1.0pt]
(m-k,m,m) & (m-k,k,0,-m) \\[-1.0pt]
(m-k,m-k,m) & (m-k,0,k,-m) \\[-1.0pt]
(0,m-k,m) & (0,m-k,k,-m) \\[-1.0pt] \hdashline
(0,m,m-k) & (0,m,-k,k-m) \\[-1.0pt]
(m,m,m-k) & (m,0,-k,k-m) \\[-1.0pt]
(m,m-k,m-k) & (m,-k,0,k-m) \\[-1.0pt]
(m,m-k,0) & (m,-k,k-m,0) \\[-1.0pt]
\hline
\end{array}
\]
Either way we find that $c_{\kappa'}(t)=0$.

{\em Case $(iii)$:} If $k>l\ge m>0$, then $(k,m,l)\sim (k-m,k+l-m,l-m)$.
The latter satisfies the conditions of Case (i) if $l>m$ and of Case (ii) if $l=m$,
so $c_{\kappa'}(t)=0$.

{\em Case $(iv)$:} If $m=k>l>0$, then $(k,k,l)\sim (l,k,0)$ which satisfies
the conditions of Case~(ii) so $c_{\kappa'}(t)=0$ in this case.

{\em Case $(v)$:}
If $k+l>m>k>l>0$, then $S_{\kappa'}$ consists of
\[\begin{array}{|c|c|}
\hline
\rho' & \rho \\ \hline
(k,m,l) & (k,m-k,l-m,-l)\\[-1.0pt]
(m-k,m,l) & (m-k,k,l-m,-l) \\[-1.0pt]
(k,k+l-m,l) & (k,l-m,m-k,-l) \\[-1.0pt] \hdashline
(l,m,k) & (l,m-l,k-m,-k)\\[-1.0pt]
(l,m,m-k) & (l,m-l,-k,k-m) \\[-1.0pt]
(l,k+l-m,k) & (l,k-m,m-l,-k)\\[-1.0pt]
\hline
\end{array}
\]
and additionally, if $m\ne 2l$,{\samepage
\[\begin{array}{|c|c|}
\hline
\rho' & \rho \\ \hline
(k,m,m-l)& (k,m-k,-l,l-m) \\[-1.0pt]
(m-k,m,m-l)& (m-k,k,-l,l-m) \\[-1.0pt]
(k,k-l,m-l)& (k,-l,m-k,l-m) \\[-1.0pt] \hdashline

(m-l,m,k)& (m-l,l,k-m,-k) \\[-1.0pt]
(m-l,m,m-k)& (m-l,l,-k,k-m) \\[-1.0pt]
(m-l,k-l,k)& (m-l,k-m,l,-k) \\[-1.0pt]
\hline
\end{array}
\]
Either way we find that $c_{\kappa'}(t)=0$.}

{\em Case $(vi)$:}
Suppose $k>m>l>0$. In this case, $S_{\kappa'}$ consists of the distinct elements among
\[\begin{array}{|c|c|}
\hline
\rho' & \rho \\ \hline
(k,m,l) & (k,m-k,l-m,-l)\\
(k,k+l-m,l) & (k,l-m,m-k,-l) \\
(k,k+l-m,k-m) & (k,l-m,-l,m-k) \\
(k,m,m-l)& (k,m-k,-l,l-m) \\
(k,k-l,m-l)& (k,-l,m-k,l-m) \\
(k,k-l,k-m)& (k,-l,l-m,m-k) \\
\hdashline

(l,m,k) & (l,m-l,k-m,-k)\\
(l,k+l-m,k) & (l,k-m,m-l,-k)\\
(k-m,k+l-m,k) & (k-m,l,m-l,-k)\\

(m-l,m,k)& (m-l,l,k-m,-k) \\
(m-l,k-l,k)& (m-l,k-m,l,-k) \\
(k-m,k-l,k)& (k-m,m-l,l,-k)\\
\hline
\end{array}
\]
There are multiplicities if $m=2l$, $m=(k+l)/2$ or $m=k-l$.
If any two of these hold then they all do and there are just two distinct elements,
namely $(3l,2l,l)$ and $(l,2l,3l)$. If just one of these holds, there are six distinct elements.
Otherwise, there are twelve. In any case, the limit~$c_{\kappa'}(t)$ is a~multiple of the
limit obtained by including all twelve, which we find to be $0$.

{\em Case $(vii)$:} $k\ge l\ge 0$ and $m=k+l$. In this case, $(k,k+l,l)\sim (k,0,l)$
which is covered by Case (ix) below.

{\em Case $(viii)$:} $k= l\ge 0$ and $m\ge 0$. If $m>k$, then $(k,m,k)\sim (m-k,0,k)$
and if $m\le k$ then~${(k,m,k)\sim (k-m,0,k)}$, both of which are covered by Case (ix) below.

{\em Case $(ix)$:} $k\ge l\ge 0$ and $m=0$.
Let us write $S_{\kappa'}=S_{k,l}$
and $c_{\kappa'}(t)=c_{k,l}(t)$. Note that in this case $P(\kappa')=k^2+l^2$.

If $k>l>0$, $S_{k,l}$ consists of the elements
$(k,0,l)$, $ (k,k+l,l)$, $ (k,k+l,k)$, $(l,0,k)$, $ (l,k+l,k)$, $ (l,k+l,l)$, $ (k,k-l,k)$.
In this case, we obtain
\[c_{k,l}(t)=8(-1)^{k+l}\big[ 1 - 2\big(k^2+l^2\big) t + \big(k^2-l^2\big)^2 t^2\big].\]
For $k>0$, $S_{k,0}$ consists of the elements
$(k,0,0)$, $ (k,k,0)$, $ (k,k,k)$, $ (0,0,k)$, $ (0,k,k)$, $ (0,k,0)$.
In this case, we obtain
$c_{k,0}(t)=4(-1)^{k}\big[ 1 - 2k^2 t + k^4 t^2\big]$.
For $k>0$, $S_{k,k}$ consists of the elements~${(k,0,k)}$, $(k,2k,k)$.
In this case, we obtain
$c_{k,k}(t)=4 \big[ 1 - 4k^2 t \big]$.
Finally, we note that~$S_{0,0}$ consists of the single element $(0,0,0)$
and in this case we have $c_{0,0}(t)=1$.

We have thus shown that
\begin{equation}\label{cp-P3} c_p(t)=\sum_{k,l\in\Z\colon k^2+l^2=p} P_{k,l}(t),\end{equation}
where
\[P_{k,l}(t) = (-1)^{k+l}\big[ 1 - 2\big(k^2+l^2\big) t + \big(k^2-l^2\big)^2 t^2\big],\]
and so the formula~\eqref{atr} yields
\[
F( t) = \sum_{k,l\in\Z} P_{k,l}(t) {\rm e}^{-(k^2+l^2)t}.
\]
This agrees with the distribution function of twice the square of the maximal height of a~pair of non-intersecting
reflected Brownian bridges, each starting and ending at zero, as discussed in Section~\ref{nibb}.

\subsection[The case r=4]{The case $\boldsymbol{r=4}$}

Again, the generating function $G_n(s)$ may be computed for small values of $n$ using the recursion~\eqref{g-rec}.
For example, using \textsc{Maple} we obtain
\begin{gather}
G_{1111}(s)=
\frac{4\big(3s^2 + 9s + 7\big)}{7(s + 1)^4(s + 2)^2},
\qquad
G_{1122}(s)=
\frac{2\big(15s^3 + 90s^2 + 167s + 100\big)}{(s + 5)^2(s + 2)^3(s + 1)^4},\label{g1122}
\end{gather}
and
\begin{equation}\label{g1234}
G_{1234}(s)=\frac{1024 J(s)}{(s + 20)(s + 8)(s + 10)^2(s + 5)^3(s + 2)^4(s + 13)(4 + s)^3(s + 1)^4},
\end{equation}
where
\begin{align*}
J(s)={}&105s^9 + 4725s^8 + 89145s^7 + 926335s^6 + 5845890s^5 + 23240520s^4\\
& + 58189896s^3 + 88421760s^2 + 73951200s + 26000000.
\end{align*}
On the other hand, using \textsc{Mathematica}, the formula~\eqref{atr-n} yields, writing $q={\rm e}^{-t}$,
\begin{gather*}
F_{1122}(t) = 1 + \frac1{48} \big(-6 + 39 \ln q - 36 (\ln q)^2 + 8 (\ln q)^3\big) q \\
 \phantom{F_{1122}(t) = }{}- \frac1{27} \big(23 - 30 \ln q + 9 (\ln q)^2\big) q^2 + \frac1{432} (-10+9\ln q) q^5,
\\
F_{1234}(t) = 1 +
 \frac{2}{729} \big(-47 - 96 \ln q - 216 (\ln q)^2 + 144 (\ln q)^3\big) q \\
 \phantom{F_{1234}(t) =}{}+ \frac1{5832} \big(-4327 + 12408 \ln q - 10080 (\ln q)^2 + 2304 (\ln q)^3\big) q^2 \\
\phantom{F_{1234}(t) =}{} - \frac1{3888} \big(993 - 1264 \ln q + 384 (\ln q)^2\big) q^4\\
\phantom{F_{1234}(t) =}{} + \frac1{729000} \big(91663 - 137160 \ln q + 50400 (\ln q)^2\big) q^5 \\
\phantom{F_{1234}(t) =}{}+ \frac{q^8}{3645} - \frac1{729000} (-203 + 240 \ln q) q^{10} - \frac{q^{13}}{204120} + \frac{q^{20}}{10206000}.
\end{gather*}
Taking Laplace transforms, we find perfect agreement with~\eqref{g1122} and~\eqref{g1234}.
We have similarly verified that the formula~\eqref{atr-n} is correct for all $n,m\le 2$, and satisfies
the boundary condition~${F_{n,m,n+m,2m}(0)=0}$ for all $n,m\le 3$ with $(n,m)\ne (0,0)$.

In all cases, we have computed, the formulas obtained are consistent with the ansatz (cf.\ \eqref{ans2} and~\eqref{ans3})
\[
F_n(t)=\sum_{k,l\in\Z} Q^n_{k,l}(t) {\rm e}^{-(k^2+l^2)t},
\]
where $Q^n_{k,l}(t)$ are polynomials of degree at most 4 (depending on $k$, $l$ only through $|k|$, $|l|$)
which satisfy
\[\frac{\rm d}{{\rm d}t} Q^n_{k,l}(t) = \big(\cL_n+k^2+l^2\big) Q^n_{k,l}(t),\qquad \sum_{k,l\in\Z} Q^n_{k,l}(0)=\delta_{n0}. \]

Now consider the formulas~\eqref{atr} and~\eqref{atr-ck}.
For $\kappa'=(k,0,0,l)$, where $k\ge l\ge 0$, denote $c_{k,l}(t)=c_{\kappa'}(t)$ and $S_{k,l}=S_{\kappa'}$.
If $k>l>0$, $S_{k,l}$ consists of 35 elements and we find, using \textsc{Mathematica},
\begin{equation}\label{ckl} c_{k,l}(t)=8Q_{k,l}(t),\end{equation}
where
\[
Q_{k,l}(t) = 1 - 4\big(k^2+l^2\big) t + 3\big(k^2+l^2\big)^2t^2
 -\frac23\big(k^2+l^2\big)^3t^3+\frac43 k^2l^2\big(k^2-l^2\big)^2 t^4.\]
If $k>0$, $S_{k,0}$ consists of 10 elements and we find that
\begin{equation}\label{ck0} c_{k,0}(t)=4Q_{k,0}(t).\end{equation}
If $k>0$, $S_{k,k}$ consists of 10 elements and we find that
\begin{equation}\label{ckk} c_{k,k}(t)=4Q_{k,k}(t).\end{equation}
Finally, we compute $c_{0,0}(t)=1$.
It appears that $c_{\kappa'}(t)=0$ for all of the other cases. Rather that go through them exhaustively
as we did in the $r=3$ case,
we note that, given~\eqref{ckl}, \eqref{ck0} and~\eqref{ckk}, this claim is equivalent to the identity
\begin{equation}\label{cp-Q} c_p(t)=\sum_{k,l\in\Z\colon k^2+l^2=p} Q_{k,l}(t),\end{equation}
which we have verified using \textsc{Mathematica} for all $p\le 12$.
Putting all this together yields the formula
\[
F( t) = \sum_{k,l\in\Z} Q_{k,l}(t) {\rm e}^{-(k^2+l^2)t},
\]
which agrees with the distribution function of twice
the square of the maximal height of a~pair of non-intersecting Brownian excursions,
as discussed in Section~\ref{nibb}.

\subsection{Higher rank}

Suppose $r=2N$ or $r=2N-1$. In the notation of Section~\ref{nibb}, for $\lambda\in\Z^N$, let $R_\lambda(t)=P_\lambda(t)$ if $r$ is odd
and $R_\lambda(t)=Q_\lambda(t)$ if $r$ is even. Note that $R_\lambda$ has degree at most $d_r$, where $d_r=(r^2-1)/4$ if $r$ is odd
and $d_r=r^2/4$ if $r$ is even.
In general, we expect the formula~\eqref{atr-n} to be consistent with the ansatz
\[
F_n(t)=\sum_{\lambda\in\Z^N} R^n_{\lambda}(t) {\rm e}^{-\sum_i \lambda_i^2t},
\]
where $R^n_{\lambda}(t)$ are polynomials of degree at most $d_r$ (depending on $\lambda$ only through $|\lambda_i|$, $i=1,\dots,N$)
which satisfy
\[\frac{\rm d}{{\rm d}t} R^n_{\lambda}(t) = \Big(\cL_n+\sum_i \lambda_i^2\Big) R^n_{\lambda}(t),\qquad \sum_{\lambda\in\Z^N} R^n_{\lambda}(0)=\delta_{n0}, \]
and $\lim_{n\to\infty}R^n_{\lambda}(t) =R_{\lambda}(t)$.

\section*{Code availability}

An open-source code repository for this work is available on GitHub~\cite{noc-gh}.

\appendix

\section{Proof of the identity~(\ref{identity})}\label{AA}

Here we give a~proof of the identity~\eqref{identity}, namely
\begin{gather*}
\sum_{k=0}^n \sum_{l=0}^m
 \frac1{\lambda_k+\mu_l}
 \prod_{\substack{a=0\\ a\ne k}}^{n} \frac{1}{(\lambda_a-\lambda_k)(\lambda_a+\lambda_k-\mu_l)(\lambda_a+\mu_l)} \\
\qquad \times \prod_{\substack{b=0\\ b\ne l}}^{m} \frac1{(\mu_b-\mu_l)(\mu_b+\mu_l-\lambda_k)(\mu_b+\lambda_k)} =0,
\end{gather*}
with the convention that empty products are equal to one and the assumption that
each term in the sum is finite.
We first note that we can write the above sum as
\begin{equation}\label{sum}
S = \sum_{l=0}^m \prod_{a=0}^n\frac1{\lambda_a+\mu_l} \prod_{\substack{b=0\\ b\ne l}}^{m} \frac1{\mu_b-\mu_l} S_l,
\end{equation}
where
\[S_l=\sum_{k=0}^n \prod_{\substack{a=0\\ a\ne k}}^{n} \frac{1}{(\lambda_a-\lambda_k)(\lambda_a+\lambda_k-\mu_l)}
\prod_{\substack{b=0\\ b\ne l}}^{m} \frac1{(\mu_b+\mu_l-\lambda_k)(\mu_b+\lambda_k)}.\]

We recall the identity, for $x_1,\dots,x_N$ distinct,
\begin{equation}\label{theta}
\sum_{i=1}^N \prod_{j\ne i} \frac1{x_j-x_i} =0.
\end{equation}
See, for example, \cite[equation (8.17)]{noc13} for a~proof.
For each $0\le l\le m$, taking $N=n+m-1$, we extend the sequence $\lambda_1,\lambda_2,\dots$
by defining $\lambda_{n+1},\dots,\lambda_N$ to be
$\mu_1+\mu_l,\dots,\mu_{l-1}+\mu_l,{\mu_{l+1}+\mu_l},\allowbreak\dots,\mu_m+\mu_l$.
Note that in this notation,
\[S_l= \sum_{k=0}^n \prod_{\substack{a=0\\ a\ne k}}^{N} \frac{1}{(\lambda_a-\lambda_k)(\lambda_a+\lambda_k-\mu_l)}.\]
Taking $x_i=\lambda_i(\lambda_i-\mu_l)$, \eqref{theta} implies that
\[\sum_{i=1}^N \prod_{j\ne i} \frac1{(\lambda_j-\lambda_i)(\lambda_j+\lambda_i-\mu_l)} =0\]
or, equivalently,
\[S_l=-\sum_{\substack{l'=0\\l'\ne l}}^m \prod_{a=0}^n\frac1{(\lambda_a-\mu_{l'}-\mu_l)(\lambda_a+\mu_{l'})}
\prod_{\substack{b=0\\ b\notin\{ l,l'\} }}^{m} \frac1{(\mu_b-\mu_{l'})(\mu_b+\mu_l+\mu_{l'})}.\]
Substituting this into~\eqref{sum} gives
\[S=\sum_{l,l'\colon l\ne l'} (\mu_{l'}-\mu_l) R_{l,l'},\]
where
\[
R_{l,l'}= \prod_{a=0}^n\frac1{(\lambda_a+\mu_l)(\lambda_a+\mu_{l'})(\lambda_a-\mu_l-\mu_{l'})} \prod_{\substack{b=0\\ b\ne l}}^{m} \frac1{\mu_b-\mu_l}
\prod_{\substack{b=0\\ b\ne l'}}^{m} \frac1{\mu_b-\mu_{l'}}
\prod_{\substack{b=0\\ b\notin\{ l,l'\} }}^{m} \frac1{\mu_b+\mu_l+\mu_{l'}}.
\]
Since $R_{l,l'}=R_{l',l}$, it follows that $S=0$, as required.

\subsection*{Acknowledgements}

The author would like to thank the anonymous referees for their careful reading and helpful comments on an earlier version of this manuscript.

\pdfbookmark[1]{References}{ref}
 \LastPageEnding

\end{document}